\documentclass[12pt]{amsart} 

\usepackage[latin1]{inputenc}
\usepackage{amsmath} 
\usepackage{amsfonts}
\usepackage{amssymb}
\usepackage{stmaryrd}
\usepackage{latexsym} 
\usepackage{graphicx}
\usepackage{subfigure}
\usepackage{color}
\usepackage{hyperref}
\usepackage{verbatim}
\usepackage[all]{xy}
\usepackage{graphics}
\usepackage{float}
\usepackage{pdfsync}
\usepackage{tikz}
\usetikzlibrary{arrows.meta}
\usepackage{url}
\usepackage{enumerate}
\usepackage{mathtools}
\usepackage{caption}

\newcommand{\svw}[1]{\textcolor{black}{#1}}

\theoremstyle{definition}
\newtheorem{theorem}{Theorem}[section]
\newtheorem*{theorem*}{Theorem}
\newtheorem{proposition}[theorem]{Proposition}

\newtheorem{lemma}[theorem]{Lemma}
\newtheorem{corollary}[theorem]{Corollary}
\newtheorem{conjecture}[theorem]{Conjecture}
\newtheorem{definition}[theorem]{Definition}
\newtheorem{example}[theorem]{Example}
 
\newtheorem{observation}[theorem]{Observation}

\theoremstyle{remark}
\newtheorem{remark}[theorem]{Remark}

\numberwithin{equation}{section}
\newcommand{\spam}{\operatorname{span}}
\newcommand{\type}{\operatorname{type}} 
\newcommand{\mt}{(\emptyset)} 

\newcommand{\Sym}{\ensuremath{\operatorname{Sym}}}

\newcommand{\pcup}{\cup}

\newcommand{\NCSym}{\ensuremath{\operatorname{NCSym}}}
\newcommand{\slashp}{\mid}

\newcommand{\NCQSym}{\ensuremath{\operatorname{NCQSym}}}

\newcommand{\UBCSym}{\ensuremath{\operatorname{UBCSym}}}
\newcommand{\y}{y}
\newcommand{\projUBC}{\nu} 
\newcommand{\projSym}{\bar\rho} 

\newcommand{\UBCQSym}{\ensuremath{\operatorname{UBCQSym}}}

\newcommand{\Sg}{\mathcal S} 
\newcommand{\sink}{\operatorname{sink}}
\newcommand{\at}{:}
\newcommand{\duct}{\mathord{\uparrow}} 
\newcommand{\Q}{Q} 
\newcommand{\set}{\operatorname{set}}

\newcommand{\eps}{\varepsilon}
\newcommand{\sepsc}{\mathbin{/\mkern-6mu/}} 

\newlength\cellsize 
\savebox2{%
\begin{picture}(15,15)
\put(0,0){\line(1,0){15}}
\put(0,0){\line(0,1){15}}
\put(15,0){\line(0,1){15}}
\put(0,15){\line(1,0){15}}
\end{picture}}
\newcommand\cellify[1]{\def\thearg{#1}\def\nothing{}%
\ifx\thearg\nothing
\vrule width0pt height\cellsize depth0pt\else
\hbox to 0pt{\usebox2\hss}\fi%
\vbox to 15\unitlength{
\vss
\hbox to 15\unitlength{\hss$#1$\hss}
\vss}}
\newcommand\tableau[1]{\vtop{\let\\=\cr
\setlength\baselineskip{-16000pt}
\setlength\lineskiplimit{16000pt}
\setlength\lineskip{0pt}
\halign{&\cellify{##}\cr#1\crcr}}}
\savebox3{%
\begin{picture}(15,15)
\put(0,0){\line(1,0){15}}
\put(0,0){\line(0,1){15}}
\put(15,0){\line(0,1){15}}
\put(0,15){\line(1,0){15}}
\end{picture}}
\newcommand\expath[1]{%
\hbox to 0pt{\usebox3\hss}%
\vbox to 15\unitlength{
\vss
\hbox to 15\unitlength{\hss$#1$\hss}
\vss}}
\newcommand\bas[1]{\omit \vbox to \cellsize{ \vss \hbox to \cellsize{\hss$#1$\hss} \vss}}

\begin{document}

\title[The chromatic symmetric function of a graph centred at a vertex]{The chromatic symmetric function\\ of a graph centred at a vertex}

\author{Farid Aliniaeifard}
\address{
 Department of Mathematics,
 University of British Columbia,
 Vancouver BC V6T 1Z2, Canada}
\email{farid@math.ubc.ca}

\author{Victor Wang}
\address{
 Department of Mathematics,
 University of British Columbia,
 Vancouver BC V6T 1Z2, Canada}
\email{vyzwang@student.ubc.ca}

\author{Stephanie van Willigenburg}
\address{
 Department of Mathematics,
 University of British Columbia,
 Vancouver BC V6T 1Z2, Canada}
\email{steph@math.ubc.ca}

\thanks{
All authors were supported in part by the National Sciences and Engineering Research Council of Canada.}

\begin{abstract} 
We discover new linear relations between the chromatic symmetric functions of certain sequences of graphs and apply these relations to find new families of $e$-positive unit interval graphs. Motivated by the results of Gebhard and Sagan, we revisit their ideas and reinterpret their equivalence relation in terms of a new quotient algebra of $\NCSym$. We investigate the projection of the chromatic symmetric function $Y_G$ in noncommuting variables in this quotient algebra, which defines $\y_{G\at v}$, the chromatic symmetric function of a graph $G$ centred at a vertex $v$. We then apply our methods to $\y_{G\at v}$ and find new families of unit interval graphs that are $(e)$-positive, a stronger condition than {classical} $e$-positivity, {thus confirming new cases of the $(3+1)$-free conjecture of Stanley and Stembridge.}

In our study of $\y_{G\at v}$, we also describe methods of constructing new $e$-positive graphs from given $(e)$-positive graphs and classify the $(e)$-positivity of trees and cut vertices. We moreover construct a related quotient algebra of $\NCQSym$ to prove theorems relating the coefficients of $\y_{G\at v}$ to acyclic orientations of graphs, including a noncommutative refinement of Stanley's sink theorem.
\end{abstract}

\maketitle
\tableofcontents

\section{Introduction}\label{sec:intro}
The chromatic symmetric function $X_G$ of a graph $G$ was introduced by Stanley \cite{Stan95} in 1995 as a generalization of Birkhoff's chromatic polynomial \cite{Birk}. Since then, it has inspired fruitful research mainly in two avenues. The first avenue is to determine whether two nonisomorphic trees can have the same chromatic symmetric function \cite{Jose2, HeilJi, LS, MMW}. Heil and Ji showed in \cite{HeilJi} that there was no counterexample on $\le29$ vertices. However, the second and more prominent avenue of research is to prove the Stanley-Stembridge conjecture \cite[Conjecture 5.5]{StanStem}, which was formulated in terms of chromatic symmetric functions by Stanley in \cite[Conjecture 5.1]{Stan95}. In 2013, Guay-Paquet \cite{MGP} showed that to prove the Stanley-Stembridge conjecture it was sufficient to prove that all unit interval graphs were $e$-positive, namely that their chromatic symmetric functions expanded with nonnegative coefficients in the basis of elementary symmetric functions.

Much interest has also arisen due to a $q$-analogue of the conjecture, introduced by Shareshian and Wachs \cite{SW} in terms of the chromatic quasisymmetric functions of labelled unit interval graphs. Brosnan and Chow \cite{BC} and, independently, Guay-Paquet \cite{MGP2} proved that the chromatic quasisymmetric functions of labelled unit interval graphs were related to the cohomology of regular semisimple Hessenberg varieties, first conjectured by Shareshian and Wachs in \cite{SW}, and this connection has been used to {prove a special case of the $q$-analogue} of the conjecture in \cite{MM}. Considerable progress has been made on the Stanley-Stembridge conjecture and Shareshian and Wachs' quasisymmetric refinement also due to the study of the modular law, which appears in several forms across \cite{AN,DvWpop,MGP,HuhNamYoo} and is a consequence of Orellana and Scott's triple-deletion rule \cite{OS} when $q=1$.

A second approach to the Stanley-Stembridge conjecture was pioneered by Gebhard and Sagan in \cite{GebSag}, where they introduced $Y_G$, the chromatic symmetric function of a graph $G$ in noncommuting variables. Intriguingly, the natural labelling of unit interval graphs was also important to their approach, as it was for chromatic quasisymmetric functions. Gebhard and Sagan proved for a large family of labelled unit interval graphs that the graphs were $(e)$-positive at their last vertices (see Definition~\ref{def:(e)pos}), a condition stronger than being $e$-positive arising from an equivalence relation on {the algebra of} $\NCSym$. Dahlberg in \cite{Dladders} proved that triangular ladders were $(e)$-positive at their last vertices using a sign-reversing involution, resolving a {special} case of the conjecture identified by Stanley in \cite{Stan95}. {Our} paper makes progress on the Stanley-Stembridge conjecture by combining the modular law with $(e)$-positivity to prove new cases of the conjecture. We also investigate in-depth the equivalence relation introduced by Gebhard and Sagan, including, in particular, the equivalence class of $Y_G$. More precisely, our paper is structured as follows.

We cover the necessary background in Section~\ref{sec:bg}. In Section~\ref{sec:arithprog} we describe in Proposition~\ref{prop:ap} sequences of graphs with the property that their chromatic symmetric functions form an arithmetic progression, which may be applied to deduce positivity. In Corollary~\ref{cor:eavg}, we specialize and obtain a tool to prove new cases of the Stanley-Stembridge conjecture, which we then apply in Proposition~\ref{prop:typeII} to prove the $e$-positivity of a new family of unit interval graphs. In Section~\ref{sec:UBCSym} we review the methods and results of Gebhard and Sagan from \cite{GebSag}, and reinterpret them in terms of a new quotient algebra {exhibiting ``unbalanced commutativity''}, $\UBCSym$, of $\NCSym$, for which we define analogues of the elementary, power sum and monomial bases. In $\UBCSym$ we also define $\y_{G\at v}$, the chromatic symmetric function of $G$ centred at $v$. We then meld the ideas of Gebhard and Sagan and the ideas of Section~\ref{sec:arithprog} to prove that many more families of labelled unit interval graphs are $(e)$-positive at their last vertices in Section~\ref{sec:new(e)}. In Section~\ref{sec:complete} we introduce a {technique} to work with linear maps on $\UBCSym$, which allows us to prove the validity of methods of constructing new $e$-positive graphs from given $(e)$-positive graphs in Theorems~\ref{the:G+Hr} and \ref{the:G-v}. In Section~\ref{sec:tree} we resolve the related questions of when trees are $(e)$-positive and which graphs are $(e)$-positive at a cut vertex. We construct a quotient algebra $\UBCQSym$ of $\NCQSym$ {in} Section~\ref{sec:UBCQSym} and prove in \svw{Theorem~\ref{the:vsink}  a noncommutative refinement of Stanley's {sink theorem} \cite[Theorem 3.3]{Stan95}.} {We conclude the paper with Section~\ref{sec:further}, in which we discuss the connections to a construction of Pawlowski \cite{Paw}, as well as possible further avenues of research.}

\section{Background}\label{sec:bg}
In this section we review the necessary background and notation that will be used in the rest of the paper. \svw{This section may be skipped, or referred back to later, by those familiar with algebraic combinatorics.}

An (integer) \textit{composition} $\alpha = (\alpha_1,\dots,\alpha_{\ell(\alpha)})$ is a finite ordered list of positive integers, where $\ell(\alpha)$ is the \textit{length} of $\alpha$. We call the integers the \textit{parts} of the composition. When $\alpha_{j+1}=\cdots=\alpha_{j+m}=i$, we often abbreviate this sublist to $i^m$. If $\alpha_1+\cdots+\alpha_{\ell(\alpha)}=d$, we say that $\alpha$ is a composition of $d$. We will also write $\emptyset$ to denote the empty composition. 

Let $[d]=\{1,\dots,d\}$. If $\alpha= (\alpha_1,\dots,\alpha_{\ell(\alpha)})$ is a composition of $d$, then we define $\text{set}(\alpha)$ to be the set $\{\alpha_1, {\alpha_1 +\alpha _2,} \dots,\alpha_1+\cdots+\alpha_{\ell(\alpha)-1}\}\subseteq [d-1]$. This induces a natural one-to-one correspondence between the compositions of $d$ and the subsets of $[d-1]$.

An (integer) \textit{partition} $\lambda = (\lambda_1,\dots,\lambda_{\ell(\lambda)})$ is a composition with {parts} satisfying $\lambda_1\ge\cdots\ge\lambda_{\ell(\lambda)}$. If $\lambda_1+\cdots+\lambda_{\ell(\lambda)}=d$, then we say that $\lambda$ is a partition of $d$ and write $\lambda\vdash d$. We also define $\lambda!$ to mean the quantity $\lambda_1!\cdots\lambda_{\ell(\lambda)}!$. {Given two partitions $\lambda$ and $\mu$, we write $\lambda\cup\mu$ to denote the partition obtained by combining the parts of $\lambda$ and $\mu$ together in weakly decreasing order.} 

We next define $\Sym$, the \textit{algebra of symmetric functions}, which may be realized as a subalgebra of $\mathbb Q[[x_1,x_2,\dots]]$, where the variables $x_j$ commute, as follows. The \textit{$i$th elementary symmetric function} $e_i$ is defined by
$$e_i = \sum_{j_1<\cdots<j_i}x_{j_1}\cdots x_{j_i}.$$
Given a partition $\lambda=(\lambda_1,\dots,\lambda_{\ell(\lambda)})$, we define the \textit{elementary symmetric function $e_\lambda$} to be
\begin{equation*}
    e_\lambda=\prod_{i=1}^{\ell(\lambda)} e_{\lambda_i}.
\end{equation*}
$\Sym$ can be defined as the graded algebra
\begin{equation*}
    \Sym=\Sym^0\oplus \Sym^1\oplus\cdots
\end{equation*}
where for each $d\in\mathbb Z_{\ge0}$, the $d$th graded piece $\Sym^d$ is spanned by the basis $\{e_\lambda\}_{\lambda\vdash d}$.

Another basis of $\Sym$ consists of the power sum symmetric functions. The \textit{$i$th power sum symmetric function} $p_i$ is defined by
$$p_i = \sum_{j}x_j^i,$$
and {we define} the \textit{power sum symmetric function $p_\lambda$} {for a partition $\lambda=(\lambda_1,\dots,\lambda_{\ell(\lambda)})$ to be}
$$p_\lambda=\prod_{i=1}^{\ell(\lambda)} p_{\lambda_i}.$$
{Then the} set $\{p_\lambda\}_{\lambda\vdash d}$ forms a basis for $\Sym^d$.

{We now turn our attention to graphs.} All graphs in this paper will be \textit{finite} and \textit{simple}. That is, our graphs $G$ will consist of a nonempty finite vertex set $V(G)$ and a finite edge set $E(G)$ consisting of pairs of distinct vertices. For $v,w\in V(G)$, we write $vw$ to mean an edge connecting $v$ and $w$. {Given a graph $G$ with vertices $v,w$, we write $G+vw$ to denote the graph obtained by adding an edge connecting $v$ and $w$ to $G$.} The \textit{order} $|G|$ of a graph $G$ is the number of vertices of the graph.

\svw{Using the standard notation,} given a vertex $v\in V(G)$, its \textit{open neighbourhood} $N(v)$ is the set of all vertices of $G$ connected by an edge to $v$. The \textit{closed neighbourhood} $N[v]$ of $v$ is the set $N(v)\cup\{v\}$.

A \textit{proper colouring} of a graph $G$ is a map $\svw{\kappa:V(G)\to\mathbb Z_{>0}}$ such that $\kappa(v)\neq \kappa(w)$ whenever $vw\in E(G)$. In 1995, Stanley defined the chromatic symmetric function of $G$ in commuting variables as follows.

\begin{definition}\label{def:XG}
\cite[Definition 2.1]{Stan95}
Let $G$ be a graph with vertex set $\{v_1,\dots,v_d\}$. Then the \textit{chromatic symmetric function} of $G$ is defined to be
\begin{equation*}
    X_G = \sum_\kappa x_{\kappa(v_1)}\cdots x_{\kappa(v_d)}
\end{equation*}
where the sum is over all proper colourings $\kappa$ of $G$.
\end{definition}

Given a basis $\{b_i\}_{i\in I}$ of a vector space \svw{over $\mathbb Q$,} we say that an element of the space is \textit{$b$-positive} if it expands in the $b$-basis with all coefficients nonnegative. We will say that a graph $G$ is {\emph{$e$-positive}} if and only if $X_G$ is $e$-positive.

A \textit{set partition} $\pi$ of $[d]$ is a collection of disjoint nonempty sets $B_1,\dots,B_{\ell(\pi)}$ whose union is $[d]$, and we denote this by
$$\pi=B_1/\cdots/B_{\ell(\pi)}\vdash[d].$$
We call the $B_j$ for $1\le j \le \ell(\pi)$ the \textit{blocks} of $\pi$ and $\ell(\pi)$ the \textit{length} of $\pi$. For ease of notation we usually omit the set parentheses and commas of the blocks. We also define $\lambda(\pi)$ to be the integer partition of $d$ whose parts are $|B_1|,\dots,|B_{\ell(\pi)}|$ sorted in weakly decreasing order. We will let $\pi!$ denote $\lambda(\pi)!$. For a set partition $\pi\vdash[d]$ we will use the notation $B_{\pi,i}$ for $i\in [d]$ to mean the block of $\pi$ containing $i$.

For a finite set of integers $S$, define $S+n=\{s+n:s\in S\}$. Then for two set partitions $\pi\vdash[n]$ and $\sigma=B_1/\cdots/B_{\ell(\sigma)}\vdash[m]$, their \textit{slash product} is defined to be
$$\pi\slashp \sigma=\pi/(B_1+n)/\cdots/(B_{\ell(\sigma)}+n)\vdash[n+m].$$

We next define $\NCSym$, the \textit{algebra of symmetric functions in noncommuting variables}, which may be realized as a subalgebra of $\mathbb Q\langle\langle x_1,x_2,\dots\rangle\rangle$, where the variables $x_j$ do not commute. $\NCSym$ can be defined as the graded algebra
$$\NCSym=\NCSym^0\oplus\NCSym^1\oplus\cdots,$$ where the $d$th graded piece is spanned by the bases $\{e_\pi\}_{\pi\vdash [d]}$, $\{p_\pi\}_{\pi\vdash [d]}$ and $\{m_\pi\}_{\pi\vdash[d]}$, which we define next.

The \textit{elementary symmetric function} $e_\pi$ in $\NCSym$ is given by $$e_\pi=\sum_{(i_1,\dots,i_d)}x_{i_1}\cdots x_{i_d},$$ summed over all tuples $(i_1,\dots,i_d)$ with $i_j\neq i_k$ if $B_{\pi,j}=B_{\pi,k}$. For set partitions $\pi,\sigma$, we have $e_\pi e_\sigma = e_{\pi\slashp\sigma},$ e.g. by \cite[Lemma 2.1]{Dladders}.

The \textit{power sum symmetric function} $p_\pi$ in $\NCSym$ is given by
$$p_\pi=\sum_{(i_1,\dots,i_d)}x_{i_1}\cdots x_{i_d},$$ summed over all tuples $(i_1,\dots,i_d)$ with $i_j= i_k$ if $B_{\pi,j}=B_{\pi,k}$.

Finally, the \textit{monomial symmetric function} $m_\pi$ in $\NCSym$ is given by
$$m_\pi= \sum_{(i_1,\dots,i_d)}x_{i_1}\cdots x_{i_d},$$ summed over all tuples $(i_1,\dots,i_d)$ with $i_j=i_k$ if and only if $B_{\pi,j}=B_{\pi,k}$.

There is an algebra map $\rho: \NCSym\to \Sym$ obtained by allowing the variables to commute. By parts (ii) and (iii) of \cite[Theorem 2.1]{RS}, we have $\rho(e_\pi)=\pi!e_{\lambda(\pi)}$ and $\rho(p_\pi)=p_{\lambda(\pi)}$.

We will also define an action of the symmetric group $\Sg_d$ on the $d$th graded piece of $\NCSym$ by permuting the positions of the variables. For $\delta\in \Sg_d$, we define the \svw{right} action on monomials by
$$\delta\circ (x_{i_1}\cdots x_{i_d})=x_{i_{\delta^{-1}(1)}}\cdots x_{i_{\delta^{-1}(d)}}$$
and extend linearly. For $\pi\vdash[d]$, we then have \svw{a left action} $\delta\circ m_\pi = m_{\delta(\pi)}$, $\delta\circ e_\pi = e_{\delta(\pi)}$ and $\delta\circ p_\pi = p_{\delta(\pi)}$ by \cite[Section 2]{GebSag}, where the action of $\delta$ on set partitions of $[d]$ is by permuting the elements of the blocks.

Gebhard and Sagan in \cite[Definition 3.4]{GebSag} defined a linear operation called \textit{induction}, $\duct$, on $\NCSym^d$ for $d\in\mathbb Z_{>0}$, by defining it on monomials via
$$(x_{i_1}\cdots x_{i_d})\duct = x_{i_1}\cdots x_{i_d}x_{i_d},$$
and extending linearly. Similarly, for $j\le d$, they also defined $\duct_j^{d+1}$ on $\NCSym^d$ by defining
$$(x_{i_1}\cdots x_{i_d})\duct_j^{d+1} = x_{i_1}\cdots x_{i_d}x_{i_j},$$ and extending linearly.

A \textit{labelled graph} on $d$ vertices is a graph with vertex set $[d]$. We can also define the action of $\delta\in\Sg_d$ on labelled graphs on $d$ vertices, by letting $\delta$ act by permuting the vertex labels. The labelled graph $\delta(G)$ is then just a relabelling of $G$. We also define \textit{reverse graph} $G^r$ of a labelled graph $G$ on $d$ vertices to be the labelled graph $\delta(G)$, where $\delta\in\Sg_d$ is the permutation exchanging each $i$ with $d+1-i$.

A \textit{labelled unit interval graph} is a labelled graph $G$ with the property that whenever $i\le v<w \le j$ and $ij\in E(G)$, then $vw\in E(G)$ as well.

\begin{definition}
Given a labelled unit interval graph on $d$ vertices, we associate two weakly increasing sequences $(m_i)_{i=1}^d$ and $(w_i)_{i=1}^d$, where $m_i\ge i$ is the largest label in $N[i]$ and $w_i\le i$ is the {smallest} label in $N[i]$.
\end{definition}

Note the closed neighbourhood $N[i]$ is given by the set $\{w_i,w_i+1,\dots,m_i\}$. Either of the two sequences is sufficient to uniquely determine the labelled unit interval graph.

Some labelled unit interval graphs we require familiarity with are the \textit{path} $P_d$ on $d$ vertices with an edge between $i$ and $i+1$ for each $i\in [d-1]$ and the \textit{complete graph} $K_d$ on $d$ vertices with an edge between every pair of distinct vertices. The cycle $C_d$ for $d\ge 3$ is obtained by adding an edge between $1$ and $d$ to the path $P_d$. We also define $K_\pi$ for $\pi\vdash[d]$ to be the labelled graph on $d$ vertices with an edge between $i\neq j$ if and only if $B_{\pi,i}=B_{\pi,j}$.

Given two labelled graphs $G$ and $H$ on $n$ and $m$ vertices, respectively, define $G\mid H$ to be the disjoint union of $G$ and $H$, where the vertices corresponding to $G$ have labels in $[n]$ in the same relative order as in $G$, and the vertices corresponding to $H$ have labels in $[m]+n$ in the same relative order as in $H$. We also define the \textit{concatenation} $G+H$ with vertex set $[n+m-1]$ to be the labelled graph obtained from $G\mid H$ by formally identifying vertices $n$ and $n+1$ of $G\mid H$ and otherwise shifting labels so that the vertices of $G+H$ have labels in the same relative order as in $G\mid H$. For a sequence $(G_j)_{j=1}^k$ of labelled graphs, define $\sum_{j=1}^kG_j$ to mean $G_1+\cdots + G_k$. If $k=0$, we take the convention that $\sum_{j=1}^kG_j = K_1$. Note when $G,H$ are labelled unit interval graphs that $G^r$, $G\mid H$ and $G+H$ are all also labelled unit interval graphs.

Guay-Paquet showed in \cite[Theorem 5.1]{MGP} that the Stanley-Stembridge conjecture is equivalent to the following.

\begin{conjecture}\label{conj:StanStem}
All labelled unit interval graphs are $e$-positive.
\end{conjecture}

Gebhard and Sagan defined a noncommutative analogue of the chromatic symmetric function in $\NCSym$, which they used to resolve cases of Conjecture~\ref{conj:StanStem}.

\begin{definition}\cite[Definition 3.1]{GebSag}
Let $G$ be a labelled graph on $d$ vertices. Then the \textit{chromatic symmetric function in noncommuting variables} of $G$ is defined to be
$$Y_G=\sum_\kappa x_{\kappa(1)}\cdots x_{\kappa(d)},$$
where the sum is over all proper colourings $\kappa$ of $G$.
\end{definition}

Note we have $\rho(Y_G)=X_G$, by \cite[Proposition 3.5]{DvWNC} $Y_{G\mid H}=Y_{G}Y_H$, by \cite[Proposition 3.3]{GebSag} $Y_{\delta(G)}=\delta\circ Y_G$, and by \cite[Lemma 4.9]{DvWNC} $Y_{K_\pi}=e_\pi$.

Gebhard and Sagan showed in \cite[Proposition 3.5]{GebSag} that $Y_G$ satisfied a \textit{deletion-contraction relation}. Given a labelled graph $G$ on $d$ vertices and an edge $jd\in E(G)$, we define $G\setminus jd$ to be the the labelled graph obtained by removing the edge $jd$ from $G$, and $G/jd$ to be the labelled graph on $d-1$ vertices obtained from $G$ by formally identifying vertices $j$ and $d$ as the single vertex $j$ in $G/jd$. Dahlberg gave in \cite[Proposition 2.2]{Dladders} a slight generalization of Gebhard and Sagan's \cite[Proposition 3.5]{GebSag}, obtained by relabelling vertices.

\begin{proposition}\cite[Proposition 2.2]{Dladders}\label{prop:delcon}
If $G$ is a labelled graph on $d$ vertices with $jd\in E(G)$, then
$$Y_G=Y_{G\setminus jd}-Y_{G/jd}\duct_j^d.$$
\end{proposition}

We end this section by defining $\NCQSym$, the \textit{algebra of quasisymmetric functions in noncommuting variables}, with bases indexed by set compositions. A \textit{set composition} $\Phi$ of $[d]$, written $\Phi\vDash[d]$, is an ordered list of blocks of some set partition $\widetilde\Phi \vdash[d]$, which we write as
$$\Phi = B_1 \sepsc \cdots \sepsc B_{\ell(\Phi)},$$
where $\ell(\Phi)$ is the \textit{length} of $\Phi$. We also define $\alpha(\Phi)$ to be the integer composition $(|B_1|,\dots,|B_{\ell(\Phi)}|)$.

The graded algebra $\NCQSym$ may be realized as a subalgebra of $\mathbb Q \langle \langle x_1,x_2,\dots\rangle\rangle$, where the variables $x_j$ do not commute, via
$$\NCQSym=\NCQSym^0\oplus\NCQSym^1\oplus\cdots,$$
where the $d$th graded piece $\NCQSym^d$ is spanned by the basis $\{M_\Phi\}_{\Phi\vDash[d]}$. The \textit{monomial quasisymmetric function} $M_\Phi$ is defined by
$$M_\Phi = \sum_{(i_1,\dots,i_d)}x_{i_1}\cdots x_{i_d},$$
summed over all tuples $(i_1,\dots,i_d)$ with $i_j=i_k$ if and only if $B_{\widetilde\Phi,j}=B_{\widetilde\Phi,k}$ and $i_j<i_k$ whenever $B_{\widetilde\Phi,j}$ appears before $B_{\widetilde\Phi,k}$ in $\Phi$. Note $\NCSym$ is a subalgebra of $\NCQSym$, via
$$m_\pi = \sum_{\widetilde\Phi = \pi} M_\Phi.$$

\section{Arithmetic progressions of graphs}\label{sec:arithprog}
Our first result describes relations between the chromatic symmetric functions of certain sequences of graphs, and has applications toward positivity.

\begin{proposition}\label{prop:ap}
Suppose $v_1,\dots,v_k$ are distinct vertices of a graph $G$ satisfying $N[v_1]=\dots=N[v_k]$. If $w$ is another vertex of $G$ not adjacent to any $v_j$, then the $(X_{G_j})_{j=0}^k$ form an arithmetic progression, where $G_0=G$ and $G_j = G_{j-1}+v_jw$ for $1\le j \le k$. In particular, if $G_0$ and $G_k$ are $b$-positive for some basis $\{b_\lambda\}_{\lambda\vdash n\ge 0}$ of $\Sym$, then so is every $G_j$ for $0\le j\le k$.
\end{proposition}

\begin{proof}
We will show that the statement of the first part of the proposition holds for any number of vertices $k$ by applying induction on $k$.

The base case $k=1$ is immediate, because $(X_{G_j})_{j=0}^1$ will always be an arithmetic progression, being a sequence of length $2$. 

When $k>1$, the vertices $v_1,v_2,w$ are mutually adjacent in $G+v_1w+v_2w$ (since $\{v_1,v_2\}\subseteq N[v_1]=N[v_2]$ implies that $v_1$ and $v_2$ are joined by an edge in $G$). Then by triple-deletion \cite[Theorem 3.1]{OS}, $$X_{G+v_1w+v_2w}=X_{G+v_1w}+X_{G+v_{2}w}-X_{G}.$$

Since $N[v_1]=N[v_2]$, we obtain a graph automorphism of $G$ by exchanging $v_1$ and $v_2$, and so $G+v_1w$ and $G+v_2w$ are isomorphic. Our equation thus becomes
$$X_{G+v_1w+v_2w}=2X_{G+v_1w}-X_{G},$$
or equivalently,
$$X_{G_2}-X_{G_1}=X_{G_1}-X_{G_0}.$$

We next note that the graph $G_1$ satisfies the hypotheses of the proposition with vertex $w$ not adjacent to the $k-1$ vertices $v_2,\dots,v_k$. By the inductive hypothesis, the $(X_{G_j})_{j=1}^k$ form an arithmetic progression. Then, since $X_{G_2}-X_{G_1}=X_{G_1}-X_{G_0}$ is its common difference, we can extend it to obtain the arithmetic progression $(X_{G_j})_{j=0}^k$.

The second part of the proposition then follows because if $X_{G_0}$ and $X_{G_k}$ are $b$-positive, then $$X_{G_j}=\frac{k-j}{k}X_{G_0}+\frac{j}{k}X_{G_k}$$
must also be $b$-positive for any $0\le j \le k$.
\end{proof}

\begin{example}
Let $G=G_0$ be the left graph in Figure~\ref{fig:schur} below. Then $G$, together with vertices $v_1,v_2,v_3$ and $w$, as labelled, satisfies the hypotheses of Proposition~\ref{prop:ap}. Let $G_1$ and $G_3$ denote $G+v_1w$ and $G+v_1w+v_2w+v_3w$, respectively.
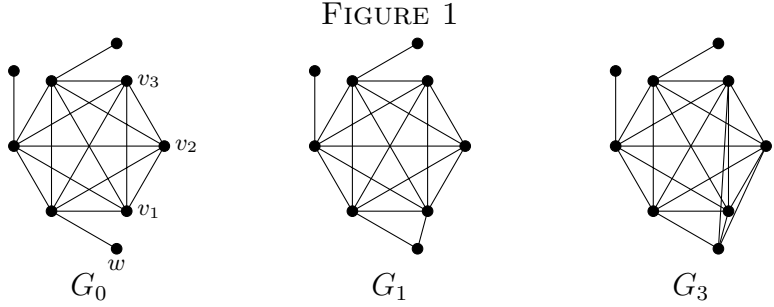
\begin{figure}[H]
\caption{}
\label{fig:schur}
\begin{tikzpicture}
\coordinate (A) at (0,0);
\coordinate (B) at (1,0);
\coordinate (C) at (1.5,.866);
\coordinate (D) at (1,1.732);
\coordinate (E) at (0,1.732);
\coordinate (F) at (-0.5,.866);
\coordinate (G) at (-0.5,1.866);
\coordinate (H) at (0.866,2.232);
\coordinate (I) at (0.866,-.5);
\draw[black] (A)--(B)--(C)--(D)--(E)--(F)--(A)--(D);
\draw[black] (E)--(A)--(C)--(E)--(B)--(D)--(F)--(B);
\draw[black] (C)--(F)--(G);
\draw[black] (I)--(A);
\draw[black] (E)--(H);
\filldraw[black] (A) circle [radius=2pt];
\filldraw[black] (B) circle [radius=2pt] node[right] {\scriptsize$v_1$};
\filldraw[black] (C) circle [radius=2pt] node[right] {\scriptsize$v_2$};
\filldraw[black] (E) circle [radius=2pt];
\filldraw[black] (G) circle [radius=2pt];
\filldraw[black] (D) circle [radius=2pt] node[right] {\scriptsize$v_3$};
\filldraw[black] (F) circle [radius=2pt];
\filldraw[black] (H) circle [radius=2pt];
\filldraw[black] (I) circle [radius=2pt] node[below] {\scriptsize$w$};
\node [] at (0.5,-1) {$G_0$};

\coordinate (A) at (4,0);
\coordinate (B) at (5,0);
\coordinate (C) at (5.5,.866);
\coordinate (D) at (5,1.732);
\coordinate (E) at (4,1.732);
\coordinate (F) at (3.5,.866);
\coordinate (G) at (3.5,1.866);
\coordinate (H) at (4.866,2.232);
\coordinate (I) at (4.866,-.5);
\draw[black] (A)--(B)--(C)--(D)--(E)--(F)--(A)--(D);
\draw[black] (E)--(A)--(C)--(E)--(B)--(D)--(F)--(B);
\draw[black] (C)--(F)--(G);
\draw[black] (I)--(A);
\draw[black] (E)--(H);
\draw[black] (I)--(B);
\filldraw[black] (A) circle [radius=2pt];
\filldraw[black] (B) circle [radius=2pt];
\filldraw[black] (C) circle [radius=2pt];
\filldraw[black] (E) circle [radius=2pt];
\filldraw[black] (G) circle [radius=2pt];
\filldraw[black] (D) circle [radius=2pt];
\filldraw[black] (F) circle [radius=2pt];
\filldraw[black] (H) circle [radius=2pt];
\filldraw[black] (I) circle [radius=2pt];
\node [] at (4.5,-1) {$G_1$};

\coordinate (A) at (8,0);
\coordinate (B) at (9,0);
\coordinate (C) at (9.5,.866);
\coordinate (D) at (9,1.732);
\coordinate (E) at (8,1.732);
\coordinate (F) at (7.5,.866);
\coordinate (G) at (7.5,1.866);
\coordinate (H) at (8.866,2.232);
\coordinate (I) at (8.866,-.5);
\draw[black] (A)--(B)--(C)--(D)--(E)--(F)--(A)--(D);
\draw[black] (E)--(A)--(C)--(E)--(B)--(D)--(F)--(B);
\draw[black] (C)--(F)--(G);
\draw[black] (I)--(A);
\draw[black] (E)--(H);
\draw[black] (I)--(B);
\draw[black] (I)--(C);
\draw[black] (I)--(D);
\filldraw[black] (A) circle [radius=2pt];
\filldraw[black] (B) circle [radius=2pt];
\filldraw[black] (C) circle [radius=2pt] ;
\filldraw[black] (E) circle [radius=2pt];
\filldraw[black] (G) circle [radius=2pt];
\filldraw[black] (D) circle [radius=2pt] ;
\filldraw[black] (F) circle [radius=2pt];
\filldraw[black] (H) circle [radius=2pt];
\filldraw[black] (I) circle [radius=2pt];
\node [] at (8.5,-1) {$G_3$};
\end{tikzpicture}
\end{figure}

Then we have, in the basis of Schur functions (see \cite[Section 1.3]{Mac} for an introduction), 
\begin{align*}
X_{G_0}&=5760s_{(1^9)} + 7200s_{(2,1^7)} + 3168s_{(2^2,1^5)} + 468s_{(2^3,1^3)} + 2880s_{(3,1^6)} + 864s_{(3,2,1^4)} + 360s_{(4,1^5)},\\
X_{G_3}&=14400s_{(1^9)} + 12960s_{(2,1^7)}+ 3888s_{(2^2,1^5)} + 288s_{(2^3,1^3)} + 2880s_{(3,1^6)} + 432s_{(3,2,1^4)},
\end{align*} which are both Schur-positive.

By Proposition~\ref{prop:ap}, the graph $G_1$ is also Schur-positive, since
\begin{align*}
    X_{G_1}&=\frac{2}{3}X_{G_0}+\frac{1}{3}X_{G_3}\\
    &=8640s_{(1^9)} + 9120s_{(2,1^7)}+ 3408s_{(2^2,1^5)} + 408s_{(2^3,1^3)} + 2880s_{(3,1^6)} + 720s_{(3,2,1^4)}+240s_{(4,1^5)}.
\end{align*}
\end{example}

We can specialise the previous proposition and obtain a tool to help prove the $e$-positivity of labelled unit interval graphs, making progress toward Conjecture~\ref{conj:StanStem}.

\begin{corollary}\label{cor:eavg}
Suppose $G$ is a labelled unit interval graph. Then the following hold.
\begin{enumerate}
    \item[(a)] If $i<|G|$ is a vertex such that $m_{i}+1\le m_{i+1}$ and $(w_{m_i+1},m_{m_{i}+1})=\cdots=(w_{m_i+k},m_{m_{i}+k})$, then the $(X_{G_j})_{j=0}^k$ form an arithmetic progression, where $G_j= G + \{ib\mid m_i+1\le b \le m_i+j\}$ for $0\le j \le k$. In particular, if $G_0$ and $G_k$ are $e$-positive, then so is every $G_j$ for $0\le j \le k$.
    \item[(b)] Alternatively, if $i>1$ is a vertex such that $w_i-1\ge w_{i-1}$ and $(w_{w_i-1},m_{w_i-1})=\cdots=(w_{w_i-k},m_{w_i-k})$, then the $(X_{G_j})_{j=0}^k$ form an arithmetic progression, where $G_j = G + \{bi\mid w_i-1\ge b\ge w_i-j\}$ for $0\le j \le k$. In particular, if $G_0$ and $G_k$ are $e$-positive, then so is every $G_j$ for $0\le j \le k$.
\end{enumerate}

\end{corollary}

\begin{proof}
Part (a) follows immediately from applying Proposition~\ref{prop:ap} to $G$ on vertices $i$ and $m_{i}+1,\dots, m_i+k$ in the $e$-basis. Part (b) follows from applying part (a) to the reverse graph $G^r$ of $G$.
\end{proof}

\begin{observation}
When the hypotheses of either part (a) or (b) of Corollary~\ref{cor:eavg} are satisfied, the labelled graphs $G_j$ for $0\le j \le k$ are labelled unit interval graphs.
\end{observation}

\begin{remark}
Special cases of the relations in Corollary~\ref{cor:eavg} were studied by Dahlberg and van Willigenburg in \cite{DvWpop} to give a proof of the $e$-positivity of the \emph{lollipop graphs} $L_{m,n}$ for $m,n\ge1$ in \cite[Theorem 8]{DvWpop}, which are the labelled unit interval graphs associated with the sequence $(w_i)_{i=1}^{m+n}$ where
$$w_i=\begin{cases}
1 &\text{if $1\le i \le m$,}\\
i-1 &\text{if $m+1\le i\le m+n$.}
\end{cases}
$$

In \cite{HuhNamYoo}, Huh, Nam and Yoo further studied these relations for the chromatic quasisymmetric function of labelled unit interval graphs, and proved in \cite[Theorem 4.9]{HuhNamYoo} the $e$-positivity of \emph{melting lollipop graphs} $L_{m,n}^{(k)}$ for $m,n\ge 1$ and $0\le k\le m-1$, obtained by deleting the edges between vertex $m$ and vertices $1,\dots,k$ from $L_{m,n}$. (The definition given differs slightly from that of Huh, Nam and Yoo's in that these are actually the reverse graphs of what they call $L_{m,n}^{(k)}$. See the rightmost graph of Figure~\ref{fig:lollipop} for the example of $L_{5,2}^{(1)}$.) In fact, Corollary~\ref{cor:eavg} is equivalent to \cite[Theorem 3.4(b')]{HuhNamYoo} for $q=1$. The quasisymmetric case of these relations is also studied in \cite{AN} by Abreu and Nigro.

We will refer to the graphs $L_{m,n}^{(k)}$ as \textit{type I melting lollipop graphs} to distinguish them from a related family of graphs introduced in the next proposition, which we will prove are $e$-positive by an application of Corollary~\ref{cor:eavg}. Melting lollipop graphs are interesting because there exists an induction scheme from which their $e$-positivity can be deduced only from the $e$-positivity argument in Corollary \ref{cor:eavg} and the $e$-positivity of disjoint unions of complete graphs.
\end{remark}

\begin{proposition}\label{prop:typeII}
\emph{Type II melting lollipop graphs} $\Gamma_{m,n}^{(k)}$ for $m\ge 3$, $n\ge 1$ and $1\le k \le m-1$, obtained by deleting the edges between vertex $1$ and vertices $m,\dots,m-k+1$ from $L_{m,n}$, are $e$-positive.

\end{proposition}
\begin{proof}
Apply Corollary~\ref{cor:eavg}(a) to the labelled unit interval graph $K_1\mid L_{m-1,n}$ on vertices $1$ and $2,\dots,m-1$. Figure~\ref{fig:lollipop} illustrates the case of $m=5$, $n=2$ and $k=2$.
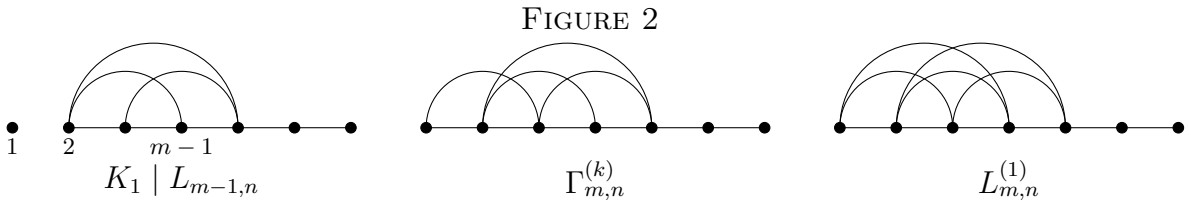
\begin{figure}[H]
\caption{}
\label{fig:lollipop}
\begin{tikzpicture}
\coordinate (A) at (.75,0);
\coordinate (B) at (1.5,0);
\coordinate (C) at (2.25,0);
\coordinate (D) at (3,0);
\coordinate (E) at (3.75,0);
\coordinate (F) at (4.5,0);
\coordinate (G) at (5.25,0);
\draw[black] (B)--(C)--(D)--(E)--(F)--(G);

\draw[black] (3,0) arc (0:180:.75);
\draw[black] (3.75,0) arc (0:180:.75);
\draw[black] (3.75,0) arc (0:180:1.125);
\filldraw[black] (A) circle [radius=2pt] node[below] {\scriptsize$1$};
\filldraw[black] (B) circle [radius=2pt] node[below] {\scriptsize$2$};
\filldraw[black] (C) circle [radius=2pt];
\filldraw[black] (E) circle [radius=2pt];
\filldraw[black] (G) circle [radius=2pt];
\filldraw[black] (D) circle [radius=2pt] node[below] {\scriptsize $m-1$};
\filldraw[black] (F) circle [radius=2pt];
\node [] at (3,-.7) {$K_1\mid L_{m-1,n}$};

\coordinate (A) at (6.25,0);
\coordinate (B) at (7,0);
\coordinate (C) at (7.75,0);
\coordinate (D) at (8.5,0);
\coordinate (E) at (9.25,0);
\coordinate (F) at (10,0);
\coordinate (G) at (10.75,0);
\draw[black] (A)--(B)--(C)--(D)--(E)--(F)--(G);
\draw[black] (7.75,0) arc (0:180:.75);
\draw[black] (8.5,0) arc (0:180:.75);
\draw[black] (9.25,0) arc (0:180:.75);
\draw[black] (9.25,0) arc (0:180:1.125);
\filldraw[black] (A) circle [radius=2pt];
\filldraw[black] (B) circle [radius=2pt];
\filldraw[black] (C) circle [radius=2pt];
\filldraw[black] (E) circle [radius=2pt];
\filldraw[black] (G) circle [radius=2pt];
\filldraw[black] (D) circle [radius=2pt];
\filldraw[black] (F) circle [radius=2pt];
\node [] at (8.5,-.7) {$\Gamma_{m,n}^{(k)}$};

\coordinate (A) at (11.75,0);
\coordinate (B) at (12.5,0);
\coordinate (C) at (13.25,0);
\coordinate (D) at (14,0);
\coordinate (E) at (14.75,0);
\coordinate (F) at (15.5,0);
\coordinate (G) at (16.25,0);
\draw[black] (A)--(B)--(C)--(D)--(E)--(F)--(G);
\draw[black] (13.25,0) arc (0:180:.75);
\draw[black] (14,0) arc (0:180:.75);
\draw[black] (14,0) arc (0:180:1.125);
\draw[black] (14.75,0) arc (0:180:.75);
\draw[black] (14.75,0) arc (0:180:1.125);
\filldraw[black] (A) circle [radius=2pt];
\filldraw[black] (B) circle [radius=2pt];
\filldraw[black] (C) circle [radius=2pt];
\filldraw[black] (E) circle [radius=2pt];
\filldraw[black] (G) circle [radius=2pt];
\filldraw[black] (D) circle [radius=2pt];
\filldraw[black] (F) circle [radius=2pt];
\node [] at (14,-.7) {$L_{m,n}^{(1)}$};
\end{tikzpicture}
\end{figure}

The graph $K_1\mid L_{m-1,n}$ is $e$-positive because lollipop graphs are $e$-positive, e.g. by \cite[Theorem 8]{DvWpop}. The type I melting lollipop graph $L_{m,n}^{(1)}$ is $e$-positive by \cite[Theorem 4.9]{HuhNamYoo}. Therefore since
$$X_{\Gamma_{m,n}^{(k)}}=\frac{k-1}{m-2}X_{K_1}X_{L_{m-1,n}}+\frac{m-k-1}{m-2} X_{L_{m,n}^{(1)}},$$
the graph $\Gamma_{m,n}^{(k)}$ is $e$-positive.
\end{proof}

\section{$\UBCSym$ and graph concatenations}\label{sec:UBCSym}
Applying Corollary~\ref{cor:eavg} to deduce the $e$-positivity of certain labelled unit interval graphs requires the $e$-positivity of a pair of labelled unit interval graphs to be known ahead of time. It will be useful to review some of the known $e$-positive labelled unit interval graphs in the literature to find more graphs on which we can apply the technique from Corollary~\ref{cor:eavg}. Gebhard and Sagan in \cite[Corollary 7.7]{GebSag} showed the $e$-positivity of all labelled unit interval graphs obtained from concatenating a sequence of complete graphs, which they called \emph{$K_\alpha$-chains}. We will also review some of the ideas they used to prove $e$-positivity, stemming from an equivalence relation in $\NCSym$.

In \cite[Section 6]{GebSag}, Gebhard and Sagan noted that ``even for some of the simplest graphs, $Y_G$ is usually not $e$-positive.'' As an example, they gave $Y_{P_3}= \frac{1}{2}e_{12/3}-\frac{1}{2}e_{13/2}+\frac{1}{2}e_{1/23}+\frac{1}{2}e_{123}$. Dahlberg and van Willigenburg later showed in \cite[Theorem 4.14]{DvWNC} that $Y_G$ is $e$-positive if and only if $G=K_\pi$ for some set partition $\pi$.

However, after defining 
$$e_{\pi_1}\equiv_i e_{\pi_2}\hbox{\rm \quad if and only if \quad}\lambda(\pi_1)=\lambda(\pi_2)\text{ and }|B_{\pi_1,i}|=|B_{\pi_2,i}|,$$
and extending linearly, we then have
$$Y_{P_3}\equiv_3 \frac{1}{2}e_{12/3}+\frac{1}{2}e_{123},$$ where the right-hand side is $e$-positive. {This is because $\lambda(13/2)=(2,1)=\lambda(1/23)$ and the size of a block containing $3$ in each of $13/2$ and $1/23$ is $2$, hence the $-\frac{1}{2}e_{13/2}$ and $+\frac{1}{2}e_{1/23}$ in $Y_{P_3}$ cancel.} Gebhard and Sagan called this relation \textit{congruence modulo $i$}.

Gebhard and Sagan in \cite{GebSag} and, later, Dahlberg in \cite{Dladders} together found several families of labelled unit interval graphs $G$ that were congruent modulo $|G|$ to an $e$-positive function in $\NCSym$. We will say then that a labelled graph $G$ is \emph{$(e)$-positive} to mean that $Y_G$ is congruent modulo $|G|$ to an $e$-positive function, e.g. the labelled graph $P_3$ is $(e)$-positive. We next give an equivalent formulation of these ideas in terms of a new quotient algebra of $\NCSym$.

For $d\in\mathbb Z_{>0}$ and any set partition $\pi\vdash [d]$, write $\type(\pi)$ to mean the pair $(\lambda,b)$, where $b=|B_{\pi,d}|$ and $\lambda$ is the partition whose parts are the sizes of the other parts of $\pi$, e.g. $${\type(1/24/35)=((2,1),2).}$$ When $\pi$ is the empty set partition, write $\type(\pi)=\mt$.

Recall that one basis of $\NCSym$ consists of the $e_\pi$ over all set partitions $\pi$. We can define $\UBCSym$ first as the free vector space spanned by elements $e_{\type(\pi)}$ over all set partitions $\pi$. Then $\UBCSym$ is naturally a quotient vector space of $\NCSym$ via the linear projection map
\begin{align*}
{\projUBC}:\NCSym &\rightarrow \UBCSym\\
e_\pi &\mapsto e_{\type(\pi)}.
\end{align*}

The kernel of $\projUBC$ is given by
$$\ker\projUBC = \spam\{e_{\pi_1}-e_{\pi_2}\mid \type(\pi_1)=\type(\pi_2)\}.$$
If set partitions $\pi_1,\pi_2$ satisfy $\type(\pi_1)=\type(\pi_2)$, then for any other set partition $\sigma$, we still have $\type(\pi_1\slashp \sigma)=\type(\pi_2\slashp \sigma)$ and $\type(\sigma\slashp\pi_1)=\type(\sigma\slashp\pi_2)$, so the kernel of $\projUBC$ is in fact a two-sided ideal of $\NCSym$, via the equalities
$$\projUBC((e_{\pi_1}-e_{\pi_2})e_\sigma) =\projUBC(e_\sigma(e_{\pi_1}-e_{\pi_2}))=0$$ and extending bilinearly.
Moreover, it is a graded ideal of $\NCSym$. This makes $\UBCSym$ a graded quotient algebra of $\NCSym$, with
$$\UBCSym = \NCSym/\spam\{e_{\pi_1}-e_{\pi_2}\mid \type(\pi_1)=\type(\pi_2)\}.$$ We will write $\UBCSym^d = \projUBC(\NCSym^d)$ to denote the homogeneous part of degree $d$ in $\UBCSym$. The kernel of $\projUBC$ is contained in the kernel of $\rho$, so the induced map $\projSym:\UBCSym\to\Sym$ is well-defined, and $\Sym$ is a quotient algebra of $\UBCSym$.

Note for $d\in\mathbb Z_{>0}$, for all $\pi\vdash [d]$ and $\delta\in \Sg_d$ fixing $d$ {we have} that $\projUBC(\delta\circ e_\pi)=\projUBC(e_{\delta(\pi)})=\projUBC(e_\pi)$. Extending linearly, for any $f\in\NCSym^d$ and $\delta\in \Sg_d$ fixing $d$, we also have $\projUBC(\delta\circ f)=\projUBC(f)$. This is the content of \cite[Lemma 6.6]{GebSag}. 

If $\type(\pi_1)=\type(\pi_2)$ for set partitions $\pi_1,\pi_2\vdash[d]$ with $d\in\mathbb Z_{>0}$, then there exists $\delta\in\Sg_d$ fixing $d$ such that $\pi_1=\delta(\pi_2)$, and so $\projUBC(p_{\pi_1})=\projUBC(p_{\pi_2})$. So we can define $p_{\type(\pi)}=\projUBC(p_\pi)$ for each set partition $\pi$. Since the $p_{\type(\pi)}$ over all $\pi\vdash [d]$ span $\UBCSym^d$, which has dimension equal to the number of distinct types of set partitions of $[d]$, it follows that the $p_{\type(\pi)}$ form another basis for $\UBCSym$. Similarly, we can define $m_{\type(\pi)}=\projUBC(m_\pi)$ for each set partition $\pi$, and they form a third basis for $\UBCSym$.

In \cite[Lemma 6.2]{GebSag}, Gebhard and Sagan state that if $f,g\in\NCSym$ are homogeneous of degree $d\in\mathbb Z_{>0}$ satisfying $\projUBC(f)=\projUBC(g)$, then $\projUBC(f\duct)=\projUBC(g\duct)$. We will define the linear operation \textit{induction}, $\duct$, on $\UBCSym^d$ to be the induced map sending $\projUBC(f)\mapsto \projUBC(f\duct)$ for every $f\in \NCSym^d$.

Our main object of study will be $\projUBC(Y_G)$ of a labelled graph $G$. Since for any labelled graph $G$ on $d$ vertices and $\delta\in \Sg_d$ fixing $d$ we have $$\projUBC(Y_{\delta(G)})=\projUBC(\delta\circ Y_G)=\projUBC(Y_G),$$
the value of $\projUBC(Y_G)$ depends only on the (unlabelled) graph $G$ and the choice of vertex labelled last.

\begin{definition}
Given a labelled graph $G$, we define
$$\y_G= \projUBC (Y_G).$$
If $G$ is a graph with a distinguished vertex $v$, the \emph{chromatic symmetric function of $G$ centred at $v$} is
$$\y_{G\at v}= \y_G,$$
where $G$ is given a labelling with $v$ as the last vertex.
\end{definition}

We will call an arbitrary function in $\UBCSym$ \emph{$(e)$-positive} if all coefficients are nonnegative in its expansion in the $e$-basis. Note that this is consistent with the notation of Gebhard and Sagan in that the following gives an equivalent definition of $(e)$-positivity of a labelled graph.

\begin{definition}\label{def:(e)pos}
A labelled graph $G$ is \emph{$(e)$-positive} if and only if $\y_G$ is $(e)$-positive. We also say that a graph $G$ is \emph{$(e)$-positive at a vertex $v$} if and only if $\y_{G\at v}$ is $(e)$-positive.
\end{definition}

As an example, we saw earlier that the labelled graph $P_3$ is $(e)$-positive, and
$$\y_{P_3}=\frac{1}{2}e_{((2),1)}+\frac{1}{2}e_{(\emptyset,3)}.$$ Note since
$$\projSym(e_{(\lambda,b)})=\lambda!b!e_{{\lambda\pcup(b)}},$$
any graph that is $(e)$-positive at some vertex is then necessarily also $e$-positive.

Gebhard and Sagan in \cite{GebSag} and Dahlberg in \cite{Dladders} found results showing for certain families of labelled graphs $H$, the concatenation $G+H$ is $(e)$-positive whenever $G$ is $(e)$-positive, motivating the following definition.

\begin{definition}
A labelled graph $H$ is \emph{appendable $(e)$-positive} if and only if $G+H$ is $(e)$-positive for all $(e)$-positive labelled graphs $G$.
\end{definition}

Appendable $(e)$-positive labelled graphs $H$ are necessarily $(e)$-positive, because $K_1$ is $(e)$-positive, and so $H=K_1+H$ must be $(e)$-positive by definition of appendable $(e)$-positivity.

We next briefly list the known $(e)$-positive and appendable $(e)$-positive labelled graphs from \cite{GebSag} and \cite{Dladders}. Results that follow by some combination of the listed propositions are omitted.

\begin{proposition}\label{prop:cycle}
\cite[Proposition 6.8]{GebSag}
Cycle graphs $C_n$ for $n\ge 3$ are $(e)$-positive.
\end{proposition}

\begin{proposition}
\cite[Theorem 7.6]{GebSag}\label{prop:Kn}
Complete graphs $K_n$ for $n\ge 1$ are appendable $(e)$-positive.
\end{proposition}

\begin{proposition}
\cite[Theorem 5.3]{Dladders}\label{prop:TL}
\emph{Triangular ladder graphs} $TL_n$ for $n\ge1$, given by the sequence $(m_i)_{i=1}^n$ where each $m_i=\min\{i+2,n\}$, are appendable $(e)$-positive.
\end{proposition}

These alone already prove the $e$-positivity of a plethora of labelled unit interval graphs. In particular, any labelled unit interval graph obtained by concatenating a sequence of complete graphs and triangular ladder graphs is $(e)$-positive (and therefore $e$-positive), as noted by Dahlberg in \cite[Corollary 5.4]{Dladders}. Dahlberg suggests in \cite[Section 5]{Dladders} that the following conjecture, which is a strengthening of the Stanley-Stembridge conjecture, may hold, and has verified it for all labelled unit interval graphs on up to $7$ vertices.

\begin{conjecture} \label{conj:alle}
\cite[Section 5]{Dladders}
All labelled unit interval graphs are $(e)$-positive.
\end{conjecture}

An even more optimistic conjecture is the following.

\begin{conjecture} \label{conj:allae}
All labelled unit interval graphs are appendable $(e)$-positive.
\end{conjecture}

One of our goals in the next section will be to lend credence to this conjecture, by finding more families of appendable $(e)$-positive labelled unit interval graphs.

\section{New $(e)$-positive labelled unit interval graphs}\label{sec:new(e)}
In this section we will combine the ideas in Section~\ref{sec:arithprog} with the ideas of Gebhard and Sagan to find new $(e)$-positive and appendable $(e)$-positive labelled unit interval graphs. We begin this section by modifying Corollary~\ref{cor:eavg} to obtain versions that apply to $(e)$-positivity and appendable $(e)$-positivity.

\begin{proposition}\label{prop:(e)avg}
Suppose $G$ is a labelled unit interval graph. Then the following hold.
\begin{enumerate}
    \item[(a)] If $i<|G|$ is a vertex such that $m_{i}+1\le m_{i+1}$ and $(w_{m_i+1},m_{m_{i}+1})=\cdots=(w_{m_i+k},m_{m_{i}+k})$ with $m_i+k<|G|$, then the $(\y_{G_j})_{j=0}^k$ form an arithmetic progression, where $G_j= G + \{ib\mid m_i+1\le b \le m_i+j\}$. In particular, if $G_0$ and $G_k$ are $(e)$-positive, then so is every $G_j$ for $0\le j \le k$.
    \item[(b)] Alternatively, if $i>1$ is a vertex such that $w_i-1\ge w_{i-1}$ and $(w_{w_i-1},m_{w_i-1})=\cdots=(w_{w_i-k},m_{w_i-k})$, then the $(\y_{G_j})_{j=0}^k$ form an arithmetic progression, where $G_j = G + \{bi\mid w_i-1\ge b\ge w_i-j\}$ for $0\le j \le k$. In particular, if $G_0$ and $G_k$ are $(e)$-positive, then so is every $G_j$ for $0\le j \le k$.
\end{enumerate}
\end{proposition}
\begin{proof}
We will only give a proof of part (a). The proof of (b) is similar.

It suffices to show that the $(\y_{G_j})_{j=0}^k$ form an arithmetic progression in $\UBCSym$, since the second part of the statement would then follow from $$\y_{G_j}=\frac{k-j}{k}\y_{G_0}+\frac{j}{k}\y_{G_k}.$$ {We will} proceed by induction on $k$.

The base case $k=1$ is immediate, because $(\y_{G_j})_{j=0}^1$ will always be an arithmetic progression, being a sequence of length $2$.

When $k>1$, note since $(w_{m_i+1},m_{m_i+1})=(w_{m_i+2},m_{m_i+2})$, there is a graph automorphism of $G$ obtained by exchanging vertices $m_i+1$ and $m_i+2$. Let $\delta \in \Sg_{|G|}$ be the permutation swapping $m_i+1$ and $m_i+2$. Then the labelled graph obtained from adding an edge joining $i$ and $m_i+2$ to $G$ is the labelled graph $\delta(G_1)$. Note $\delta$ fixes $|G|$, since $m_i+k<|G|$.

The vertices $m_i+1$, $m_i+2$, $i$ are mutually adjacent in $G_2$. By noncommutative triple-deletion \cite[Proposition 3.6]{DvWNC}, $$Y_{G_2}-Y_{\delta(G_1)}-Y_{G_1}+Y_{G_0}=0.$$ Applying the projection map $\projUBC$, we obtain $$\y_{G_2}-2\y_{G_1}+\y_{G_0}=0.$$

We next note that the labelled unit interval graph $G_1$ satisfies the hypotheses of part (a) with vertex $i$ and the $k-1$ vertices $m_{i}+2,\dots,m_i+k$ (where $m_i$ is defined by the sequence of $G$). By the inductive hypothesis, the $(\y_{G_j})_{j=1}^k$ form an arithmetic progression with common difference $\y_{G_2}-\y_{G_1}=\y_{G_1}-\y_{G_0}$. Therefore it can be extended to obtain the arithmetic progression $(\y_{G_j})_{j=0}^k$, as desired.
\end{proof}

\begin{corollary}\label{cor:a(e)avg}
Suppose $G$ is a labelled unit interval graph. Then the following hold.
\begin{enumerate}
    \item[(a)] If $i<|G|$ is a vertex such that $m_{i}+1\le m_{i+1}$ and $(w_{m_i+1},m_{m_{i}+1})=\cdots=(w_{m_i+k},m_{m_{i}+k})$ with $m_i+k<|G|$, then define the labelled graphs $G_j= G + \{ib\mid m_i+1\le b \le m_i+j\}$ for $0\le j \le k$. If $G_0$ and $G_k$ are appendable $(e)$-positive, then so is every $G_j$ for $0\le j \le k$.
    \item[(b)] Alternatively, if $i>1$ is a vertex such that $w_i-1\ge w_{i-1}$ and $(w_{w_i-1},m_{w_i-1})=\cdots=(w_{w_i-k},m_{w_i-k})$ with $w_i-k>1$, then define the labelled graphs $G_j = G + \{bi\mid w_i-1\ge b\ge w_i-j\}$ for $0\le j \le k$. If $G_0$ and $G_k$ are appendable $(e)$-positive, then so is every $G_j$ for $0\le j \le k$.
\end{enumerate}
\end{corollary}
\begin{proof}
For part (a), note if $G_0$ and $G_k$ are appendable $(e)$-positive, then for every $n\in \mathbb Z_{>0}$, the labelled graphs $K_n+G_0$ and $K_n+G_k$ are $(e)$-positive. The labelled unit interval graph $K_n+G_0$ satisfies the hypotheses of Proposition~\ref{prop:(e)avg}(a) on the images in $K_n+G_0$ of the vertices $i$ and $m_i+1,\dots,m_i+k$ of $G_0$. Therefore, 
$$\y_{K_n+G_j}=\frac{k-j}{k}\y_{K_n+G_0} + \frac{j}{k}\y_{K_n+G_k},$$
and so $K_n+G_j$ is $(e)$-positive for every $n\in \mathbb Z_{>0}$ and every $0\le j\le k$. By Proposition~\ref{prop:Kd+H}, the labelled graphs $G_j$ for $0\le j\le k$ are all appendable $(e)$-positive.

The proof of part (b) is the same, except part (b) of Proposition~\ref{prop:(e)avg} is applied instead of part (a).
\end{proof}

\begin{remark}\label{rem:rev}
Note the symmetry between the statements of parts (a) and (b) of Corollary~\ref{cor:a(e)avg}. In particular, if the labelled unit interval graphs $(G_j)_{j=0}^k$ are considered in either part of the corollary, then the appendable $(e)$-positivity of the $(G_j^r)_{j=0}^k$ follow from the appendable $(e)$-positivity of $G_0^r$ and $G_k^r$ by applying the opposite part of the corollary.
\end{remark}

We next prove the $(e)$-positivity and appendable $(e)$-positivity of several families of labelled unit interval graphs. The appendable $(e)$-positivity of our first family of graphs will follow from the computations of Gebhard and Sagan in \cite[Section 7]{GebSag}.

\begin{proposition}\label{prop:TP}
\emph{Twin peaks graphs} $TP_{n+1}$ for $n\ge 2$, obtained by removing the edge between $1$ and $n+1$ in $K_{n+1}$, are appendable $(e)$-positive.
\end{proposition}
\begin{proof}
Let $G$ be a labelled graph with $$\y_G=\sum_{|\lambda|+b=|G|}c_{(\lambda,b)}e_{(\lambda,b)}.$$ Then by \cite[Lemma 7.3]{GebSag} and \cite[Lemma 7.5]{GebSag}, which describe the expansions of $\y_{G+K_{n+1}}$ and $\y_{G+K_n:|G|}\duct$ in the $e$-basis,
\begin{align*}
    \y_{G+K_{n+1}}&=\sum_{|\lambda|+b=|G|}\sum_{i=0}^{n-1} c_{(\lambda,b)} \frac{(n-1)!(b-1)!}{(n-i-1)!(b+i)!}\left((b-n+i)e_{({\lambda\pcup(b+i)},n-i)} + (i+1)e_{({\lambda\pcup(n-i-1)},b+i+1)}\right),\\
\y_{G+K_n\at |G|}\duct &= \sum_{|\lambda|+b=|G|}\sum_{i=0}^{n-1}c_{(\lambda,b)}\frac{(n-1)!(b-1)!}{(n-i-1)!(b+i)!}\left(e_{({\lambda\pcup(b+i)},n-i)} -e_{({\lambda\pcup(n-i-1)},b+i+1)}\right).
\end{align*}

By the deletion-contraction relation in Proposition~\ref{prop:delcon},
\begin{align*}
    \y_{G+TP_{n+1}}&=\y_{G+K_{n+1}}+\y_{G+K_n\at |G|}\duct\\
    &=\sum_{|\lambda|+b=|G|}\sum_{i=0}^{n-1} c_{(\lambda,b)} \frac{(n-1)!(b-1)!}{(n-i-1)!(b+i)!}\left((b-n+i+1)e_{({\lambda\pcup(b+i)},n-i)} + ie_{({\lambda\pcup(n-i-1)},b+i+1)}\right).
\end{align*}

Now suppose $G$ is $(e)$-positive, i.e. each coefficient $c_{(\lambda,b)}$ is $\ge 0$. The contribution of $c_{(\lambda,b)}e_{(\lambda,b)}$ to $\y_{G+TP_{n+1}}$ is $$\sum_{i=0}^{n-1} c_{(\lambda,b)} \frac{(n-1)!(b-1)!}{(n-i-1)!(b+i)!}\left((b-n+i+1)e_{({\lambda\pcup(b+i)},n-i)} + ie_{({\lambda\pcup(n-i-1)},b+i+1)}\right),$$ which has nonnegative coefficients in the $e$-basis, except possibly at the $e_{({\lambda\pcup(b+i)},n-1)}$ when both $0\le i\le n-1$ and $b-n+i+1<0$. In that case, $j=n-i-b-1$ satisfies $0\le j\le n-1$ and $({\lambda\pcup(n-j-1)},b+j+1)= ({\lambda\pcup(b+i)},n-i)$.

So the coefficient of $e_{({\lambda\pcup(b+i)},n-i)}$ in the contribution of $c_{(\lambda,b)}e_{(\lambda,b)}$ to $\y_{G+TP_{n+1}}$ when $0\le i \le n-1$ and $b-n+i+1<0$ is
$$c_{(\lambda,b)}\left(\frac{(n-1)!(b-1)!}{(n-i-1)!(b+i)!}(b-n+i+1)+\frac{(n-1)!(b-1)!}{(b+i)!(n-i-1)!}(n-i-b-1)\right)=0.$$

Therefore, if $G$ is $(e)$-positive, then so is $G+TP_{n+1}$. {That is,} $TP_{n+1}$ is appendable $(e)$-positive.
\end{proof}

We can now find new $(e)$-positive and appendable $(e)$-positive families of graphs by applying Proposition~\ref{prop:(e)avg} and Corollary~\ref{cor:a(e)avg} to all the known $(e)$-positive and appendable $(e)$-positive labelled unit interval graphs we already have.

\begin{proposition}\label{prop:IC}
\emph{Melting ice cream scoop graphs} $IC_{n+1}^{(k)}$ for $n\ge 2$ and $1\le k \le n$, obtained by deleting the edges between vertex $n+1$ and vertices $1,\dots,k$ from $K_{n+1}$, and their reverse graphs are appendable $(e)$-positive.
\end{proposition}
\begin{proof}
To prove the appendable $(e)$-positivity of $IC_{n+1}^{(k)}$, {we will} apply Corollary~\ref{cor:a(e)avg}(b) to $K_n\mid K_1$ on vertex $n+1$ and vertices $n,\dots,2$. Figure~\ref{fig:IC} illustrates the case of $n=5$ and $k=2$.
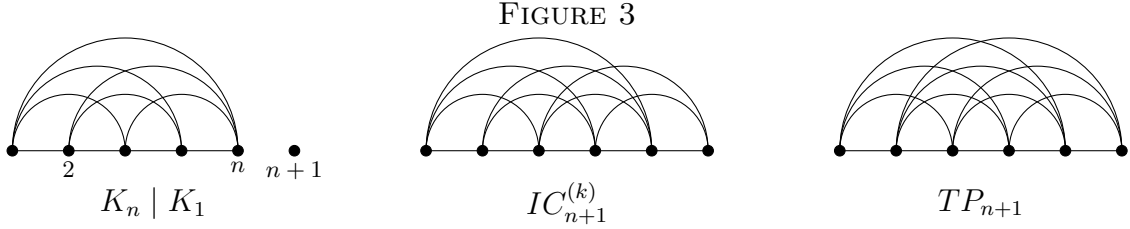
\begin{figure}[H]
\caption{}
\label{fig:IC}
\begin{tikzpicture}
\coordinate (A) at (.75,0);
\coordinate (B) at (1.5,0);
\coordinate (C) at (2.25,0);
\coordinate (D) at (3,0);
\coordinate (E) at (3.75,0);
\coordinate (F) at (4.5,0);
\draw[black] (A)--(B)--(C)--(D)--(E);
\filldraw[black] (A) circle [radius=2pt];
\filldraw[black] (B) circle [radius=2pt] node[below] {\scriptsize$2$};
\filldraw[black] (C) circle [radius=2pt];
\filldraw[black] (E) circle [radius=2pt]  node[below] {\scriptsize $n$};
\filldraw[black] (D) circle [radius=2pt];
\filldraw[black] (F) circle [radius=2pt] node[below] {\scriptsize $n+1$};
\node [] at (2.625,-.7) {$K_n\mid K_1$};
\draw[black] (2.25,0) arc (0:180:.75);
\draw[black] (3,0) arc (0:180:.75);
\draw[black] (3,0) arc (0:180:1.125);
\draw[black] (3.75,0) arc (0:180:.75);
\draw[black] (3.75,0) arc (0:180:1.125);
\draw[black] (3.75,0) arc (0:180:1.5);

\coordinate (A) at (6.25,0);
\coordinate (B) at (7,0);
\coordinate (C) at (7.75,0);
\coordinate (D) at (8.5,0);
\coordinate (E) at (9.25,0);
\coordinate (F) at (10,0);
\draw[black] (A)--(B)--(C)--(D)--(E)--(F);
\filldraw[black] (A) circle [radius=2pt];
\filldraw[black] (B) circle [radius=2pt];
\filldraw[black] (C) circle [radius=2pt];
\filldraw[black] (E) circle [radius=2pt]; 
\filldraw[black] (D) circle [radius=2pt];
\filldraw[black] (F) circle [radius=2pt];
\node [] at (8.125,-.7) {$IC_{n+1}^{(k)}$};
\draw[black] (7.75,0) arc (0:180:.75);
\draw[black] (8.5,0) arc (0:180:.75);
\draw[black] (8.5,0) arc (0:180:1.125);
\draw[black] (9.25,0) arc (0:180:.75);
\draw[black] (9.25,0) arc (0:180:1.125);
\draw[black] (9.25,0) arc (0:180:1.5);
\draw[black] (10,0) arc (0:180:.75);
\draw[black] (10,0) arc (0:180:1.125);

\coordinate (A) at (11.75,0);
\coordinate (B) at (12.5,0);
\coordinate (C) at (13.25,0);
\coordinate (D) at (14,0);
\coordinate (E) at (14.75,0);
\coordinate (F) at (15.5,0);
\draw[black] (A)--(B)--(C)--(D)--(E)--(F);
\filldraw[black] (A) circle [radius=2pt];
\filldraw[black] (B) circle [radius=2pt];
\filldraw[black] (C) circle [radius=2pt];
\filldraw[black] (E) circle [radius=2pt]; 
\filldraw[black] (D) circle [radius=2pt];
\filldraw[black] (F) circle [radius=2pt];
\draw[black] (13.25,0) arc (0:180:.75);
\draw[black] (14,0) arc (0:180:.75);
\draw[black] (14,0) arc (0:180:1.125);
\draw[black] (14.75,0) arc (0:180:.75);
\draw[black] (14.75,0) arc (0:180:1.125);
\draw[black] (14.75,0) arc (0:180:1.5);
\draw[black] (15.5,0) arc (0:180:.75);
\draw[black] (15.5,0) arc (0:180:1.125);
\draw[black] (15.5,0) arc (0:180:1.5);
\node [] at (13.625,-.7) {$TP_{n+1}$};
\end{tikzpicture}
\end{figure}

Note $K_n\mid K_1$ is appendable $(e)$-positive because $K_n$ is appendable $(e)$-positive by Proposition~\ref{prop:Kn} and because for any labelled graph $G$ we have $\y_{G+K_n\mid K_1}=\y_{G+K_n}e_{(\emptyset,1)}$. Additionally, $TP_{n+1}$ is appendable $(e)$-positive by Proposition~\ref{prop:TP}. Therefore, $IC_{n+1}^{(k)}$ is appendable $(e)$-positive by Corollary~\ref{cor:a(e)avg}(b) {setting $G_0=K_n\mid K_1$, $G_{n-1}=TP_{n+1}$ and $i=n+1$.}

Next note that $K_1\mid K_n$, the reverse graph of $K_n\mid K_1$, is appendable $(e)$-positive, because for any labelled graph $G$ we have $\y_{G+K_1\mid K_n}=\y_Ge_{(\emptyset,n)}$. The reverse graph of $TP_{n+1}$, which is $TP_{n+1}$ itself, is appendable $(e)$-positive by Proposition~\ref{prop:TP}. By Remark~\ref{rem:rev}, the reverse graph of $IC_{n+1}^{(k)}$ is also appendable $(e)$-positive.
\end{proof}

\begin{proposition}
\emph{Snowy twin peaks graphs}\label{prop:STP} $STP_{n+1}^{k}$ for $n\ge 3$ and $1\le k\le n-2$, given by the sequence $(m_i)_{i=1}^{n+1}$ where
$$m_i=\begin{cases}k+1&\text{if }i=1,\\
n&\text{if }i=2,\\
n+1&\text{if }3\le i \le n+1,
\end{cases}$$
and their reverse graphs are appendable $(e)$-positive.
\end{proposition}
\begin{proof}
To prove the appendable $(e)$-positivity of $STP_{n+1}^k$, {we will} apply Corollary~\ref{cor:a(e)avg}(a) to $K_2+TP_n$ on vertex $1$ with vertices $3,\dots,n$. Figure~\ref{fig:STP} illustrates the case of $n=6$ and $k=3$.
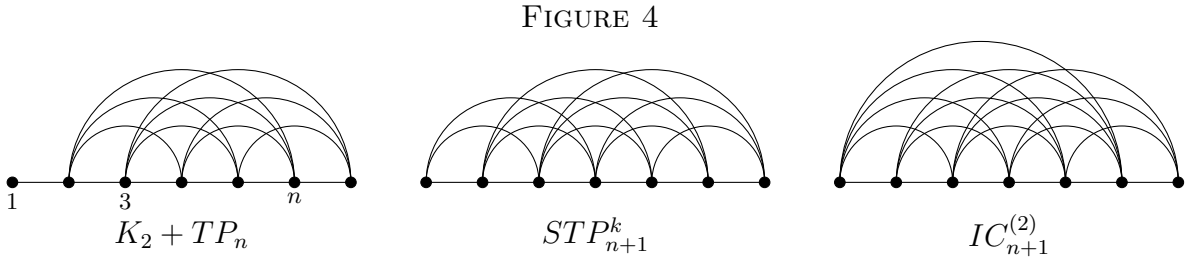
\begin{figure}[H]
\caption{}
\label{fig:STP}
\begin{tikzpicture}
\coordinate (A) at (.75,0);
\coordinate (B) at (1.5,0);
\coordinate (C) at (2.25,0);
\coordinate (D) at (3,0);
\coordinate (E) at (3.75,0);
\coordinate (F) at (4.5,0);
\coordinate (G) at (5.25,0);
\draw[black] (A)--(B)--(C)--(D)--(E)--(F)--(G);

\draw[black] (3,0) arc (0:180:.75);
\draw[black] (3.75,0) arc (0:180:.75);
\draw[black] (3.75,0) arc (0:180:1.125);
\draw[black] (4.5,0) arc (0:180:.75);
\draw[black] (4.5,0) arc (0:180:1.125);
\draw[black] (4.5,0) arc (0:180:1.5);
\draw[black] (5.25,0) arc (0:180:.75);
\draw[black] (5.25,0) arc (0:180:1.125);
\draw[black] (5.25,0) arc (0:180:1.5);
\filldraw[black] (A) circle [radius=2pt] node[below] {\scriptsize$1$};
\filldraw[black] (B) circle [radius=2pt];
\filldraw[black] (C) circle [radius=2pt] node[below] {\scriptsize $3$};
\filldraw[black] (E) circle [radius=2pt] ;
\filldraw[black] (G) circle [radius=2pt];
\filldraw[black] (D) circle [radius=2pt];
\filldraw[black] (F) circle [radius=2pt] node[below] {\scriptsize $n$};
\node [] at (3,-.7) {$K_{2}+TP_n$};

\coordinate (A) at (6.25,0);
\coordinate (B) at (7,0);
\coordinate (C) at (7.75,0);
\coordinate (D) at (8.5,0);
\coordinate (E) at (9.25,0);
\coordinate (F) at (10,0);
\coordinate (G) at (10.75,0);
\draw[black] (A)--(B)--(C)--(D)--(E)--(F)--(G);
\draw[black] (7.75,0) arc (0:180:.75);
\draw[black] (8.5,0) arc (0:180:.75);
\draw[black] (9.25,0) arc (0:180:.75);
\draw[black] (9.25,0) arc (0:180:1.125);
\draw[black] (10,0) arc (0:180:.75);
\draw[black] (10,0) arc (0:180:1.125);
\draw[black] (10,0) arc (0:180:1.5);
\draw[black] (10.75,0) arc (0:180:.75);
\draw[black] (10.75,0) arc (0:180:1.125);
\draw[black] (10.75,0) arc (0:180:1.5);
\draw[black] (8.5,0) arc (0:180:1.125);
\filldraw[black] (A) circle [radius=2pt];
\filldraw[black] (B) circle [radius=2pt];
\filldraw[black] (C) circle [radius=2pt];
\filldraw[black] (E) circle [radius=2pt];
\filldraw[black] (G) circle [radius=2pt];
\filldraw[black] (D) circle [radius=2pt];
\filldraw[black] (F) circle [radius=2pt];
\node [] at (8.5,-.7) {$STP^k_{n+1}$};

\coordinate (A) at (11.75,0);
\coordinate (B) at (12.5,0);
\coordinate (C) at (13.25,0);
\coordinate (D) at (14,0);
\coordinate (E) at (14.75,0);
\coordinate (F) at (15.5,0);
\coordinate (G) at (16.25,0);
\draw[black] (A)--(B)--(C)--(D)--(E)--(F)--(G);
\draw[black] (13.25,0) arc (0:180:.75);
\draw[black] (14,0) arc (0:180:.75);
\draw[black] (14.75,0) arc (0:180:.75);
\draw[black] (14.75,0) arc (0:180:1.125);
\draw[black] (15.5,0) arc (0:180:.75);
\draw[black] (15.5,0) arc (0:180:1.125);
\draw[black] (15.5,0) arc (0:180:1.5);
\draw[black] (15.5,0) arc (0:180:1.875);
\draw[black] (16.25,0) arc (0:180:.75);
\draw[black] (16.25,0) arc (0:180:1.125);
\draw[black] (16.25,0) arc (0:180:1.5);
\draw[black] (14.75,0) arc (0:180:1.5);
\draw[black] (14,0) arc (0:180:1.125);
\filldraw[black] (A) circle [radius=2pt];
\filldraw[black] (B) circle [radius=2pt];
\filldraw[black] (C) circle [radius=2pt];
\filldraw[black] (E) circle [radius=2pt];
\filldraw[black] (G) circle [radius=2pt];
\filldraw[black] (D) circle [radius=2pt];
\filldraw[black] (F) circle [radius=2pt];
\node [] at (14,-.7) {$IC_{n+1}^{(2)}$};
\end{tikzpicture}
\end{figure}

Note $K_2+TP_n$ is appendable $(e)$-positive because both $K_2$ and $TP_n$ are appendable $(e)$-positive by Propositions~\ref{prop:Kn} and \ref{prop:TP}, respectively, and so for any $(e)$-positive labelled graph $G$, the labelled graph $G+K_2$ is $(e)$-positive, and therefore the labelled graph $G+K_2+TP_n$ is $(e)$-positive. Additionally, $IC_{n+1}^{(2)}$ is appendable $(e)$-positive by Proposition~\ref{prop:IC}. Therefore, $STP_{n+1}^k$ is appendable $(e)$-positive {by applying Corollary~\ref{cor:a(e)avg}(a) setting $G_0=K_2+TP_n$, $G_{n-2}=IC_{n+1}^{(2)}$ and $i=1$.}

Next note that $TP_n+K_2$, the reverse graph of $K_2+TP_n$, is appendable $(e)$-positive because $TP_n$ and $K_2$ are appendable $(e)$-positive by Propositions~\ref{prop:TP} and \ref{prop:Kn}, and the reverse graph of $IC^{(2)}_{n+1}$ is appendable $(e)$-positive by Proposition~\ref{prop:IC}. By Remark~\ref{rem:rev}, the reverse graph of $STP_{n+1}^k$ is also appendable $(e)$-positive.
\end{proof}

\begin{proposition}
\emph{Wide melting lollipop graphs} $WL_{m,n}^{(k)}$ for $m\ge 4$, $n\ge 0$ and $1\le k \le m-2$, given by the sequence $(w_i)_{i=1}^{m+n}$ where
$$w_i=\begin{cases}
1&\text{if $1\le i\le m-1,$}\\
k+1&\text{if $i=m,$}\\
i-2&\text{if $m+1\le i\le m+n,$}
\end{cases}$$
are $(e)$-positive.
\end{proposition}
\begin{proof}
{We will} proceed by induction on $n$. For the base case $n=0$, note {that in Proposition~\ref{prop:IC} we can equivalently define $IC_m^{(k)}$ to be given by the sequence $(w_i)_{i=1}^m$ where}
$${w_i=\begin{cases}1&\text{if }1\le i \le m-1,\\ k+1&\text{if }i=m.
\end{cases}}$$
{Hence} $WL_{m,0}^{(k)}=IC_m^{(k)}$, which is $(e)$-positive by Proposition~\ref{prop:IC}.

For $n>0$, {we will} apply Corollary~\ref{cor:a(e)avg}(b) to $K_{m-1}+TL_{{n+2}}$ on vertex $m$ with vertices $m-2,\dots,1$. Figure~\ref{fig:WL} illustrates the case of $m=5$, $n=2$ and $k=1$.

\begin{figure}[H]
\caption{}
\label{fig:WL}
\begin{tikzpicture}
\coordinate (A) at (.75,0);
\coordinate (B) at (1.5,0);
\coordinate (C) at (2.25,0);
\coordinate (D) at (3,0);
\coordinate (E) at (3.75,0);
\coordinate (F) at (4.5,0);
\coordinate (G) at (5.25,0);
\draw[black] (A)--(B)--(C)--(D)--(E)--(F)--(G);

\draw[black] (2.25,0) arc (0:180:.75);
\draw[black] (3,0) arc (0:180:.75);
\draw[black] (4.5,0) arc (0:180:.75);
\draw[black] (5.25,0) arc (0:180:.75);
\draw[black] (3,0) arc (0:180:1.125);
\filldraw[black] (A) circle [radius=2pt] node[below] {\scriptsize$1$};
\filldraw[black] (B) circle [radius=2pt];
\filldraw[black] (C) circle [radius=2pt] node[below] {\scriptsize $m-2$};
\filldraw[black] (E) circle [radius=2pt]  node[below] {\scriptsize $m$};
\filldraw[black] (G) circle [radius=2pt];
\filldraw[black] (D) circle [radius=2pt];
\filldraw[black] (F) circle [radius=2pt];
\node [] at (3,-.7) {$K_{m-1}+TL_{{n+2}}$};

\coordinate (A) at (6.25,0);
\coordinate (B) at (7,0);
\coordinate (C) at (7.75,0);
\coordinate (D) at (8.5,0);
\coordinate (E) at (9.25,0);
\coordinate (F) at (10,0);
\coordinate (G) at (10.75,0);
\draw[black] (A)--(B)--(C)--(D)--(E)--(F)--(G);
\draw[black] (7.75,0) arc (0:180:.75);
\draw[black] (8.5,0) arc (0:180:.75);
\draw[black] (10,0) arc (0:180:.75);
\draw[black] (10.75,0) arc (0:180:.75);
\draw[black] (9.25,0) arc (0:180:.75);
\draw[black] (9.25,0) arc (0:180:1.125);
\draw[black] (8.5,0) arc (0:180:1.125);
\filldraw[black] (A) circle [radius=2pt];
\filldraw[black] (B) circle [radius=2pt];
\filldraw[black] (C) circle [radius=2pt];
\filldraw[black] (E) circle [radius=2pt];
\filldraw[black] (G) circle [radius=2pt];
\filldraw[black] (D) circle [radius=2pt];
\filldraw[black] (F) circle [radius=2pt];
\node [] at (8.5,-.7) {$WL_{m,n}^{(k)}$};

\coordinate (A) at (11.75,0);
\coordinate (B) at (12.5,0);
\coordinate (C) at (13.25,0);
\coordinate (D) at (14,0);
\coordinate (E) at (14.75,0);
\coordinate (F) at (15.5,0);
\coordinate (G) at (16.25,0);
\draw[black] (A)--(B)--(C)--(D)--(E)--(F)--(G);
\draw[black] (13.25,0) arc (0:180:.75);
\draw[black] (14,0) arc (0:180:.75);
\draw[black] (15.5,0) arc (0:180:.75);
\draw[black] (16.25,0) arc (0:180:.75);
\draw[black] (14.75,0) arc (0:180:.75);
\draw[black] (14.75,0) arc (0:180:1.125);
\draw[black] (14.75,0) arc (0:180:1.5);
\draw[black] (14,0) arc (0:180:1.125);
\filldraw[black] (A) circle [radius=2pt];
\filldraw[black] (B) circle [radius=2pt];
\filldraw[black] (C) circle [radius=2pt];
\filldraw[black] (E) circle [radius=2pt];
\filldraw[black] (G) circle [radius=2pt];
\filldraw[black] (D) circle [radius=2pt];
\filldraw[black] (F) circle [radius=2pt];
\node [] at (14,-.7) {$WL_{m+1,n-1}^{(m-2)}$};
\end{tikzpicture}
\end{figure}

Note $K_{m-1}+TL_{{n+2}}$ is $(e)$-positive because triangular ladders are appendable $(e)$-positive by Proposition~\ref{prop:TL}, and $WL_{m+1,n-1}^{(m-2)}$ is $(e)$-positive by the inductive hypothesis. Therefore since
$$\y_{WL_{m,n}^{(k)}} = \frac{k}{m-2}\y_{K_{m-1}+TL_{{n+2}}} + \frac{m-k-2}{m-2}\y_{WL_{m+1,n-1}^{(m-2)}},$$
the labelled graph $WL_{m,n}^{(k)}$ is $(e)$-positive.
\end{proof}

\begin{proposition}
Wide melting lollipop graphs $WL_{m,1}^{(k)}$ with $n=1$, for $m\ge 4$ and $1\le k \le m-2$, and their reverse graphs are appendable $(e)$-positive.
\end{proposition}
\begin{proof}
To prove the appendable $(e)$-positivity of $WL_{m,1}^{(k)}$, {we will} apply Corollary~\ref{cor:a(e)avg}(b) to $K_{m-1}+K_3$ on vertex $m$ and vertices $m-2,\dots,2$. Figure~\ref{fig:WL1} illustrates the case of $m=6$ and $k=2$.
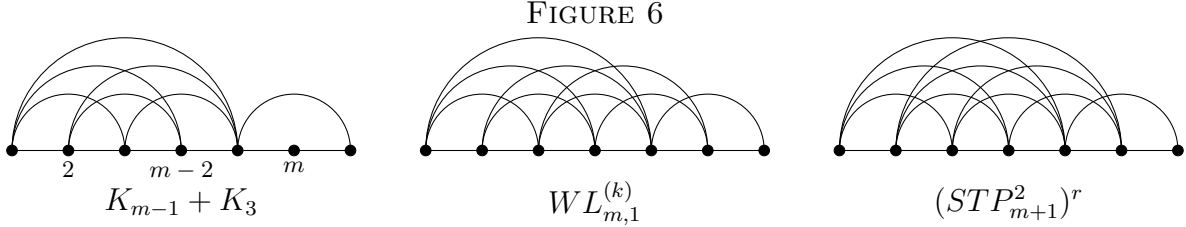
\begin{figure}[H]
\caption{}
\label{fig:WL1}
\begin{tikzpicture}
\coordinate (A) at (.75,0);
\coordinate (B) at (1.5,0);
\coordinate (C) at (2.25,0);
\coordinate (D) at (3,0);
\coordinate (E) at (3.75,0);
\coordinate (F) at (4.5,0);
\coordinate (G) at (5.25,0);
\draw[black] (A)--(B)--(C)--(D)--(E)--(F)--(G);

\draw[black] (2.25,0) arc (0:180:.75);
\draw[black] (3,0) arc (0:180:.75);
\draw[black] (5.25,0) arc (0:180:.75);
\draw[black] (3.75,0) arc (0:180:.75);
\draw[black] (3.75,0) arc (0:180:1.125);
\draw[black] (3.75,0) arc (0:180:1.5);
\draw[black] (3,0) arc (0:180:1.125);
\filldraw[black] (A) circle [radius=2pt];
\filldraw[black] (B) circle [radius=2pt] node[below] {\scriptsize$2$};
\filldraw[black] (C) circle [radius=2pt];
\filldraw[black] (E) circle [radius=2pt] ;
\filldraw[black] (G) circle [radius=2pt];
\filldraw[black] (D) circle [radius=2pt] node[below] {\scriptsize $m-2$};
\filldraw[black] (F) circle [radius=2pt] node[below] {\scriptsize $m$};
\node [] at (3,-.7) {$K_{m-1}+K_3$};

\coordinate (A) at (6.25,0);
\coordinate (B) at (7,0);
\coordinate (C) at (7.75,0);
\coordinate (D) at (8.5,0);
\coordinate (E) at (9.25,0);
\coordinate (F) at (10,0);
\coordinate (G) at (10.75,0);
\draw[black] (A)--(B)--(C)--(D)--(E)--(F)--(G);
\draw[black] (7.75,0) arc (0:180:.75);
\draw[black] (8.5,0) arc (0:180:.75);
\draw[black] (10,0) arc (0:180:.75);
\draw[black] (10,0) arc (0:180:1.125);
\draw[black] (10.75,0) arc (0:180:.75);
\draw[black] (9.25,0) arc (0:180:.75);
\draw[black] (9.25,0) arc (0:180:1.125);
\draw[black] (9.25,0) arc (0:180:1.5);
\draw[black] (8.5,0) arc (0:180:1.125);
\filldraw[black] (A) circle [radius=2pt];
\filldraw[black] (B) circle [radius=2pt];
\filldraw[black] (C) circle [radius=2pt];
\filldraw[black] (E) circle [radius=2pt];
\filldraw[black] (G) circle [radius=2pt];
\filldraw[black] (D) circle [radius=2pt];
\filldraw[black] (F) circle [radius=2pt];
\node [] at (8.5,-.7) {$WL_{m,1}^{(k)}$};

\coordinate (A) at (11.75,0);
\coordinate (B) at (12.5,0);
\coordinate (C) at (13.25,0);
\coordinate (D) at (14,0);
\coordinate (E) at (14.75,0);
\coordinate (F) at (15.5,0);
\coordinate (G) at (16.25,0);
\draw[black] (A)--(B)--(C)--(D)--(E)--(F)--(G);
\draw[black] (13.25,0) arc (0:180:.75);
\draw[black] (14,0) arc (0:180:.75);
\draw[black] (15.5,0) arc (0:180:.75);
\draw[black] (15.5,0) arc (0:180:1.125);
\draw[black] (15.5,0) arc (0:180:1.5);
\draw[black] (16.25,0) arc (0:180:.75);
\draw[black] (14.75,0) arc (0:180:.75);
\draw[black] (14.75,0) arc (0:180:1.125);
\draw[black] (14.75,0) arc (0:180:1.5);
\draw[black] (14,0) arc (0:180:1.125);
\filldraw[black] (A) circle [radius=2pt];
\filldraw[black] (B) circle [radius=2pt];
\filldraw[black] (C) circle [radius=2pt];
\filldraw[black] (E) circle [radius=2pt];
\filldraw[black] (G) circle [radius=2pt];
\filldraw[black] (D) circle [radius=2pt];
\filldraw[black] (F) circle [radius=2pt];
\node [] at (14,-.7) {$(STP^2_{m+1})^r$};
\end{tikzpicture}
\end{figure}

Note $K_{m-1}+K_3$ is appendable $(e)$-positive because both $K_{m-1}$ and $K_3$ are appendable $(e)$-positive by Proposition~\ref{prop:Kn}, and the reverse graph of $STP_{m+1}^2$ is appendable $(e)$-positive by Proposition~\ref{prop:STP}. Therefore $WL_{m,1}^{(k)}$ is appendable $(e)$-positive {by Corollary~\ref{cor:a(e)avg} setting $G_0=K_{m-1}+K_3$, $G_{m-3}=(STP^2_{m+1})^r$ and $i=m$.}

Next note that $K_3+K_{m-1}$, the reverse graph of $K_{m-1}+K_3$, is appendable $(e)$-positive because both $K_3$ and $K_{m-1}$ are appendable $(e)$-positive by Proposition~\ref{prop:Kn}, and $STP_{m+1}^2$, the reverse graph of $(STP_{m+1}^2)^r$, is appendable $(e)$-positive by Proposition~\ref{prop:STP}. By Remark~\ref{rem:rev}, the reverse graph of $WL_{m,1}^{(k)}$ is also appendable $(e)$-positive.
\end{proof}

\section{The reduction to complete graphs}
\label{sec:complete}
In this section we will highlight a simple but useful idea. For any linear map on $\UBCSym^d$ with a special interpretation when $\y_{G\at v}$ is {substituted} in, we can compute the image of any function in $\UBCSym^d$ by writing the function in {the} $e$-basis and then noting each $e_{(\lambda,b)}=\y_{K_{\lambda_1}\mid\cdots\mid K_{\lambda_{\ell(\lambda)}}\mid K_b}$. In particular, to understand the map on $\UBCSym^d$, it is enough to study its behaviour on disjoint unions of complete graphs.

Our first illustration of this concept is a shorter derivation of the first part of \cite[Corollary 6.1]{GebSag}, which computes the result of the induction operation on the $e$-basis of $\UBCSym$. The second part of \cite[Corollary 6.1]{GebSag} also follows by symmetry.

\begin{lemma}\cite[Corollary 6.1]{GebSag}\label{lem:duct}
For any partition $\lambda$ and positive integer $b$,
$$e_{(\lambda,b)}\duct = \frac{1}{b}e_{({\lambda\pcup(b)},1)}-\frac{1}{b}e_{(\lambda,b+1)}.$$
\end{lemma}
\begin{proof}
By Proposition~\ref{prop:delcon}, deletion-contraction, $$\y_{K_{\lambda_1}\mid\cdots\mid K_{\lambda_{\ell(\lambda)}}\mid K_b+K_2}=\y_{K_{\lambda_1}\mid\cdots\mid K_{\lambda_{\ell(\lambda)}}\mid K_b\mid K_1}-\y_{K_{\lambda_1}\mid\cdots\mid K_{\lambda_{\ell(\lambda)}}\mid K_b}\duct.$$ Therefore,
$$e_{(\lambda,b)}\duct=\y_{K_{\lambda_1}\mid\cdots\mid K_{\lambda_{\ell(\lambda)}}\mid K_b}\duct=\y_{K_{\lambda_1}\mid\cdots\mid K_{\lambda_{\ell(\lambda)}}\mid K_b\mid K_1} - \y_{K_{\lambda_1}\mid\cdots\mid K_{\lambda_{\ell(\lambda)}}\mid K_b+K_2}.$$ To compute $\y_{K_{\lambda_1}\mid\cdots\mid K_{\lambda_{\ell(\lambda)}}\mid K_b+K_2}=e_{((\lambda_1,\dots,\lambda_{\ell(\lambda)-1}),\lambda_{\ell(\lambda)})} \y_{K_b+K_2}$, we note by Proposition~\ref{prop:(e)avg}(b) (applied to $K_b\mid K_1$ and vertex $b+1$ together with vertices $b,\dots,1$) that $$\y_{K_b+K_2}= \frac{b-1}{b}\y_{K_b\mid K_1}+\frac{1}{b}\y_{K_{b+1}}=\frac{b-1}{b}e_{((b),1)}+\frac{1}{b}e_{(\emptyset,b+1)}.$$

{Hence we} conclude that
\begin{align*}
e_{(\lambda,b)}\duct&=\y_{K_{\lambda_1}\mid\cdots\mid K_{\lambda_{\ell(\lambda)}}\mid K_b\mid K_1} - \y_{K_{\lambda_1}\mid\cdots\mid K_{\lambda_{\ell(\lambda)}}\mid K_b+K_2}\\
&=e_{(\lambda\pcup (b),1)}-\left( \frac{b-1}{b}e_{({\lambda\pcup (b)},1)}+\frac{1}{b}e_{(\lambda,b+1)} \right) \\
&=\frac{1}{b}e_{(\lambda\pcup(b),1)}-\frac{1}{b}e_{(\lambda,b+1)}.
\end{align*}\end{proof}

To give more applications of the idea introduced in this section, we will require more linear maps on $\UBCSym^d$ with special interpretations when $\y_{G\at v}$ is {substituted} in. Our next result shows for any labelled graph $H$ that the map sending each $\y_G\mapsto \y_{G+H}$ is linear, which also explains why graph concatenation is natural to study in $\UBCSym$.

\begin{theorem}\label{the:linear}
If $H$ is a labelled graph, then for every $d\in\mathbb Z_{>0}$, there exists a linear map $\overline{T_H}:\UBCSym^d\to \UBCSym^{d+|H|-1}$ sending each $\y_G\mapsto \y_{G+H}$ for all labelled graphs $G$ on $d$ vertices.
\end{theorem}
\begin{proof}
For each subset $S_2\subseteq E(H)$, define $T_{S_2}:\NCSym^{d} \to \UBCSym^{d+|H|-1}$ to be the linear map taking each $p_\pi \mapsto p_{\type(\pi')}$, where given a set partition $\pi\vdash [d]$, we construct the set partition $\pi' \vdash [d+|H|-1]$ describing the connected components of the labelled graph $K_\pi+(V(H),S_2)$.

We next show that $T_{S_2}$ induces a linear map on $\UBCSym^d$. If $\pi_1,\pi_2\vdash [d]$ are of the same type, then there exists $\delta\in \Sg_d$ fixing $d$ such that $\pi_1=\delta(\pi_2)$. If we interpret $\delta$ as an element of $\Sg_{d+|H|-1}$ {acting only on} the first $d$ positions, we see that $\delta$ fixes the last $|H|$ positions and satisfies $\pi_1'=\delta(\pi_2')$. So $\type(\pi_1')=\type(\pi_2')$ and $$T_{S_2}(p_{\pi_1})=p_{\type(\pi_1')}=p_{\type(\pi_2')}=T_{S_2}(p_{\pi_2})$$ whenever $\pi_1,\pi_2$ are of the same type. This induces a linear map $\overline{T_{S_2}}:\UBCSym^d\to\UBCSym^{d+|H|-1}$.

Next we define $$\overline{T_H}=\sum_{S_2\subseteq E(H)}(-1)^{|S_2|}\overline{T_{S_2}}.$$

Let $G$ be any labelled graph on $d$ vertices. By \cite[{Theorem 3.6}]{GebSag},
\begin{align*}
    Y_{G+H}&=\sum_{S\subseteq E(G+H)}(-1)^{|S|}p_{\pi(S)}\\
    &=\sum_{S_1\subseteq E(G)} (-1)^{|S_1|}\sum_{S_2\subseteq E(H)} (-1)^{|S_2|}p_{\pi(S_1\cup S_2)},
\end{align*}
where $\pi(S)$ is the set partition of $[d+|H|-1]$ describing the connected components of the labelled graph $(V(G+H),S)$. Also, $$Y_G=\sum_{S_1\subseteq E(G)}(-1)^{|S_1|}p_{\pi(S_1)},$$ where in this case $\pi(S_1)\vdash [d]$ describes the connected components of $(V(G),S_1)$. 

Then
\begin{align*}
    \overline{T_H}(\y_G) &= \sum_{S_2\subseteq E(H)}(-1)^{|S_2|}\overline{T_{S_2}}(\y_G) \\
    &= \sum_{S_2\subseteq E(H)}(-1)^{|S_2|}{T_{S_2}}\left(\sum_{S_1\subseteq E(G)}(-1)^{|S_1|}p_{\pi(S_1)}\right)\\
    &= \sum_{S_2\subseteq E(H)} (-1)^{|S_2|} \sum_{S_1\subseteq E(G)} (-1)^{|S_1|} p_{\type(\pi(S_1 \cup S_2))}=\y_{G+H},
\end{align*}
proving the statement of the proposition.
\end{proof}

We can apply our {techniques} to the map in Theorem~\ref{the:linear} to obtain an equivalent condition to appendable $(e)$-positivity.

\begin{proposition} \label{prop:Kd+H}
A labelled graph $H$ is appendable $(e)$-positive if and only if $K_d+H$ is $(e)$-positive for all $d\in \mathbb Z_{>0}$.
\end{proposition}
\begin{proof}
The forward direction follows immediately from the definition of appendable $(e)$-positivity, since each $K_d$ is $(e)$-positive.

For the reverse direction, suppose each $K_d+H$ is $(e)$-positive. Let $G$ be any $(e)$-positive labelled graph. By Theorem~\ref{the:linear},
\begin{align*}
    \y_{G+H} &=\overline{T_H}(\y_G)\\
    &= \overline{T_H}\left(\sum_{|\lambda|+b=|G|}c_{(\lambda,b)}e_{(\lambda,b)}\right)\\
    &= \sum_{|\lambda|+b=|G|}c_{(\lambda,b)}\overline{T_H}(e_{(\lambda,b)})\\
    &= \sum_{|\lambda|+b=|G|}c_{(\lambda,b)}\overline{T_H}(\y_{K_{\lambda_1}\mid\cdots\mid K_{\lambda_{\ell(\lambda)}}\mid K_b}) \\
    &= \sum_{|\lambda|+b=|G|}c_{(\lambda,b)}\y_{K_{\lambda_1}\mid\cdots\mid K_{\lambda_{\ell(\lambda)}}\mid K_b+H} \\
    &= \sum_{|\lambda|+b=|G|}c_{(\lambda,b)} e_{((\lambda_1,\dots,\lambda_{\ell(\lambda)-1}),\lambda_{\ell(\lambda)})} \y_{K_b+H}
\end{align*}
for some nonnegative coefficients $c_{(\lambda,b)}$, and so $\y_{G+H}$ is $(e)$-positive. Since this holds for all $(e)$-positive $G$, it follows that $H$ is appendable $(e)$-positive.
\end{proof}

\begin{corollary}
Conjectures~\ref{conj:alle} and~\ref{conj:allae} are equivalent.
\end{corollary}
\begin{proof}
Conjecture~\ref{conj:allae} implies Conjecture~\ref{conj:alle}, because every appendable $(e)$-positive labelled graph is $(e)$-positive.

For the other direction, suppose Conjecture~\ref{conj:alle} held. Let $H$ be any labelled unit interval graph. For each $d\in\mathbb Z_{>0}$, the labelled graph $K_d+H$ is also a labelled unit interval graph, and so is $(e)$-positive, by assumption. Then by Proposition~\ref{prop:Kd+H}, every labelled unit interval graph $H$ is appendable $(e)$-positive. So Conjecture~\ref{conj:allae} would follow from Conjecture~\ref{conj:alle}.
\end{proof}

Another application of the idea in this section gives a method of constructing new $e$-positive graphs from a pair of $(e)$-positive graphs.

\begin{theorem}\label{the:G+Hr}
If $G$ and $H$ are $(e)$-positive labelled graphs, then $G+H^r$ is $e$-positive.
\end{theorem}
\begin{proof}
By Theorem~\ref{the:linear}, there exists a linear map $\overline{T_{H^r}}:\UBCSym^{|G|}\to \UBCSym^{|G|+|H|-1}$ sending $\y_{G'}\mapsto \y_{G'+H^r}$ for all labelled graphs $G'$ on $|G|$ vertices.

Since $G$ is $(e)$-positive, there exist nonnegative coefficients $c_{(\lambda,b)}$ such that
\begin{align*}
    X_{G+H^r} &= \projSym\left(\overline{T_{H^r}}\left(\y_G\right)\right) \\
    &= \projSym\left(\overline{T_{H^r}}\left(\sum_{|\lambda|+b=|G|}c_{(\lambda,b)}e_{(\lambda,b)}\right)\right) \\
    &= \sum_{|\lambda|+b=|G|}c_{(\lambda,b)}\projSym\left(\overline{T_{H^r}}(e_{(\lambda,b)})\right) \\
    &= \sum_{|\lambda|+b=|G|}c_{(\lambda,b)}\projSym\left(\overline{T_{H^r}}(\y_{K_{\lambda_1}\mid \dots \mid K_{\lambda_{\ell(\lambda)}} \mid K_b})\right) \\ 
    &= \sum_{|\lambda|+b=|G|}c_{(\lambda,b)}X_{K_{\lambda_1}\mid \dots \mid K_{\lambda_{\ell(\lambda)}} \mid K_b+H^r}\\
    &= \sum_{|\lambda|+b=|G|}c_{(\lambda,b)}\lambda! e_\lambda X_{H+K_b}.\\
\end{align*}

Note since $H$ is $(e)$-positive, by Proposition~\ref{prop:Kn} each $H+K_b$ is $(e)$-positive (and hence $e$-positive). So $G+H^r$ is $e$-positive.
\end{proof}

We next give a summary of the positivity results we obtain from concatenating $(e)$-positive and appendable $(e)$-positive graphs.

\begin{corollary}\label{cor:summ}
If $G,G'$ are $(e)$-positive and $(H_i)_{i=1}^k$ are appendable $(e)$-positive, then
\begin{enumerate}
    \item[(a)]
    $\sum_{i=1}^k H_i$ is appendable $(e)$-positive,
    \item[(b)] $G+\sum_{i=1}^kH_i$ is $(e)$-positive, and
    \item[(c)] $G+\sum_{i=1}^kH_i+G'^{r}$ is $e$-positive.
\end{enumerate}
\end{corollary}
\begin{proof}
Part (a) follows from definition of appendable $(e)$-positivity and induction on $k$. Part (b) follows from part (a). Part (c) follows from part (b) together with Theorem~\ref{the:G+Hr}.
\end{proof}
We demonstrate Corollary~\ref{cor:summ} in the following proposition, which proves the $e$-positivity of \textit{kayak paddle graphs} $KP_{m,\ell-1,n}=C_m+P_{\ell+1}+C_n$ for $m,n\ge3$ and $\ell\ge 0$. See Figure~\ref{fig:KP} below for the example of $KP_{4,3,4}$.

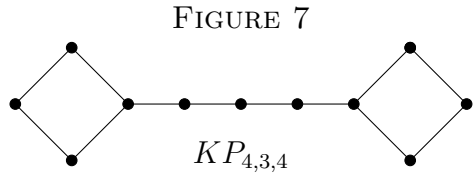
\begin{figure}[H]
    
    \caption{}
    \label{fig:KP}
    \begin{tikzpicture}
    \coordinate (A) at (6.25,0.75);
    \coordinate (B) at (7,0);
    \coordinate (C) at (7.75,0);
    \coordinate (D) at (8.5,0);
    \coordinate (E) at (9.25,0);
    \coordinate (F) at (10,0);
    \coordinate (G) at (10.75,.75);
    \coordinate (H) at (11.5,0);
    \coordinate (I) at (10.75,-0.75);
    \coordinate (J) at (5.5,0);
    \coordinate (K) at (6.25,-.75);
    \draw[black] (B)--(K)--(J)--(A)--(B)--(C)--(D)--(E)--(F)--(G)--(H)--(I)--(F);
    \filldraw[black] (A) circle [radius=2pt];
    \filldraw[black] (B) circle [radius=2pt];
    \filldraw[black] (C) circle [radius=2pt];
    \filldraw[black] (E) circle [radius=2pt];
    \filldraw[black] (G) circle [radius=2pt];
    \filldraw[black] (D) circle [radius=2pt];
    \filldraw[black] (F) circle [radius=2pt];
    \filldraw[black] (G) circle [radius=2pt];
    \filldraw[black] (H) circle [radius=2pt];
    \filldraw[black] (I) circle [radius=2pt];
    \filldraw[black] (J) circle [radius=2pt];
    \filldraw[black] (K) circle [radius=2pt];
    \node [] at (8.5,-.7) {$KP_{4,3,4}$};
    \end{tikzpicture}
\end{figure}

\begin{proposition}
Kayak paddle graphs $KP_{m,\ell-1,n}$ for $m,n\ge3$ and $\ell\ge 0$, are $e$-positive.
\end{proposition}
\begin{proof}
Note $KP_{m,\ell-1,n}=C_m+\sum_{i=1}^\ell K_2 + C_n^r$, where the cycles $C_m$ and $C_n$ are $(e)$-positive by Proposition~\ref{prop:cycle}, and the complete graph $K_2$ is appendable $(e)$-positive by Proposition~\ref{prop:Kn}. By Corollary~\ref{cor:summ}(c), the graph $KP_{m,\ell-1,n}$ is $e$-positive.
\end{proof}

Our last result of the section is a theorem relating the $(e)$-positivity of $\y_{G\at v}$ to the $e$-positivity of $X_{G-v}$, {where $G-v$ is the graph $G$ with vertex $v$ and all incident edges deleted.}

\begin{theorem}\label{the:G-v}
If a graph $G$ is $(e)$-positive at a vertex $v$, then $G-v$ is $e$-positive.
\end{theorem}
\begin{proof}
Define the linear map $\vartheta: \UBCSym^{|G|}\to \Sym^{|G|-1}$ satisfying
$$\vartheta (p_{(\lambda,b)}) = \begin{cases}
p_\lambda &\text{if }b=1,\\
0&\text{otherwise.}
\end{cases}$$

Give $G$ a labelling in which $v$ is labelled last. By \cite[{Theorem 3.6}]{GebSag},
$$\y_{G\at v}= \sum_{S\subseteq E(G)} (-1)^{|S|}p_{\type(\pi(S))},$$
where $\pi(S)$ is the set partition describing the connected components of $(V(G),S)$. Note $\type(\pi(S))=(\lambda,b)$ has $b=1$ if and only if $S\subseteq E(G-v)$. So
$$\y_{G\at v}=\sum_{S\subseteq E(G-v)}(-1)^{|S|}p_{(\lambda(S),1)} + \sum_{\substack{S\subseteq E(G)\\ S\not\subseteq E(G-v)}} (-1)^{|S|}p_{\type(\pi(S))},$$
where for $S\subseteq E(G-v)$ the integer partition $\lambda(S)$ describes the connected components of $(V(G-v),S)$, and ${\type(\pi(S))}=(\lambda,b)$ has $b>1$ for all $S\subseteq E(G)$ satisfying $S\not\subseteq E(G-v)$.

Therefore,
$$\vartheta(\y_{G\at v}) = \sum_{S\subseteq E(G-v)}(-1)^{|S|}p_{\lambda(S)}= X_{G-v},$$ with the right equality {following} by \cite[Theorem 2.5]{Stan95}. {The equation above} holds for all graphs on $|G|$ vertices with a distinguished vertex so {in particular,}
$$\vartheta(e_{(\lambda,b)})=\vartheta(\y_{K_{\lambda_1}\mid\dots\mid K_{\ell(\lambda)}\mid K_b})=X_{K_{\lambda_1}\mid\cdots\mid K_{\lambda_{\ell(\lambda)}}\mid K_{b-1}} = \lambda!(b-1)!e_{{\lambda\pcup(b-1)}},$$
where for a labelled graph $G$ we take $G\mid K_0$ to mean $G$, and for a partition $\lambda$ we take $\lambda\pcup(0)$ to mean $\lambda$.

So if
$$\y_{G\at v}=\sum_{|\lambda|+b=|G|}c_{(\lambda,b)}e_{(\lambda,b)},$$
then we have
$$X_{G-v}=\vartheta(\y_{G\at v})=\sum_{|\lambda|+b=|G|}c_{(\lambda,b)}\lambda!(b-1)!e_{{\lambda\pcup(b-1)}}.$$ In particular if $G$ is $(e)$-positive at $v$, then $G-v$ is $e$-positive.
\end{proof}

\begin{example}
Let $G$ denote the left graph in Figure~\ref{fig:G-v} below, with vertex $v$ as labelled.

\begin{figure}[H]
\caption{} 
\label{fig:G-v}
\begin{tikzpicture}
\coordinate (A) at (0,0);
\coordinate (B) at (1,0);
\coordinate (C) at (2,0);
\coordinate (D) at (3,0);
\coordinate (E) at (1,1);
\coordinate (F) at (2,1);
\draw[black] (A)--(B)--(C)--(D);
\draw[black] (B)--(E)--(F)--(C);

\filldraw[black] (A) circle [radius=2pt];
\filldraw[black] (B) circle [radius=2pt];
\filldraw[black] (C) circle [radius=2pt];
\filldraw[black] (E) circle [radius=2pt] node[above] {\scriptsize $v$};
\filldraw[black] (D) circle [radius=2pt];
\filldraw[black] (F) circle [radius=2pt];
\node [] at (1.5,-.5) {$G$};

\coordinate (A) at (5,0);
\coordinate (B) at (6,0);
\coordinate (C) at (7,0);
\coordinate (D) at (8,0);
\coordinate (F) at (7,1);
\filldraw[black] (A) circle [radius=2pt];
\filldraw[black] (B) circle [radius=2pt];
\filldraw[black] (C) circle [radius=2pt];
\filldraw[black] (D) circle [radius=2pt];
\filldraw[black] (F) circle [radius=2pt];
\draw[black] (A)--(B)--(C)--(D);
\draw[black] (F)--(C);
\node [] at (6.5,-.5) {$G-v$};
\end{tikzpicture}
\end{figure}

Then
$$\y_{G\at v}=\frac{1}{12}e_{((4,1),1)}+\frac{1}{60}e_{((5),1)}+\frac{1}{4}e_{((2^2),2)}+\frac{1}{6}e_{((3,1),2)}+\frac{1}{24}e_{((4),2)}+\frac{1}{6}e_{((1^2),4)}+\frac{1}{12}e_{((2),4)}+ \frac{1}{6}e_{((1),5)}+ \frac{1}{40}e_{(\emptyset, 6)} ,$$
which is $(e)$-positive. By Theorem~\ref{the:G-v}, the graph $G-v$ is therefore $e$-positive.
\end{example}

\section{The $(e)$-positivity of trees and cut vertices}
\label{sec:tree}
In \cite{DSvWcut}, Dahlberg, She and van Willigenburg studied the positivity of chromatic symmetric functions of trees in the Schur and $e$-bases. It will be of interest to study which trees are $(e)$-positive and at which vertices, because we can construct more $e$-positive trees from given $(e)$-positive trees by applying Theorem~\ref{the:G+Hr} or the appendable $(e)$-positivity of paths.

In their study of positivity of trees, it was also important for Dahlberg, She and van Willigenburg to understand how cut vertices affect positivity. It is also natural to ask if and when a graph can be $(e)$-positive at a cut vertex, noting that none of the $(e)$-positivity results from the previous sections demonstrate $(e)$-positivity at a cut vertex. In this short section, we will resolve both questions.

\begin{proposition}\label{prop:cut}
If $G$ is a graph with cut vertex $v$, then $G$ is not $(e)$-positive at $v$.
\end{proposition}
\begin{proof}
Since $v$ is a cut vertex of $G$, the graph $G-v$ has at least $2$ connected components, each of order $<|G|-1$.

Give $G$ a labelling where $v$ is ordered last. Consider the graph $G+G^r$ on $2|G|-1$ vertices. The image of $v$ in $G+G^r$ is a cut vertex whose deletion gives the disjoint union of $2$ copies of $G-v$, the connected components of which each have order $<|G|-1=\lfloor \frac{2|G|-1}{2}\rfloor$. By \cite[Theorem 35]{DSvWcut}, the graph $G+G^r$ is not $e$-positive.

Then by Theorem~\ref{the:G+Hr}, the graph $G$ cannot be $(e)$-positive at cut vertex $v$.
\end{proof}

\begin{remark}
In Section~\ref{sec:UBCSym}, we defined appendable $(e)$-positivity only for labelled graphs, although there is a notion of appendable $(e)$-positivity that depends only on a graph and a pair of distinguished vertices. (Namely, we may want to ask for which graphs $H$ with distinguished vertices $v,w$ is it true that for all graphs $G$ $(e)$-positive at a vertex $u$ the concatenation obtained by identifying vertex $u$ of $G$ and vertex $v$ of $H$ is $(e)$-positive at $w$.) In our definition of appendable $(e)$-positivity for labelled graphs, these vertices are the first and last vertices of the labelled graph. However, for graphs on $\ge 2$ vertices, there is no labelling that allows a single vertex to be both the first and last vertex of the labelling. This raises the possibility that there may exist nontrivial cases in which a graph and a pair of distinguished vertices satisfies the more general notion of appendable $(e)$-positivity, but in a way that cannot be encapsulated by the notion of appendable $(e)$-positivity for labelled graphs.

Proposition~\ref{prop:cut} {combined with Proposition~\ref{prop:Kd+H}} resolves this concern, because for any connected labelled graph $G$ on $\ge 2$ vertices, $K_d+G$ for $d\ge 2$ cannot be $(e)$-positive at the cut vertex $d$. So every connected graph with a pair of distinguished vertices that can be considered appendable $(e)$-positive admits a labelling in which it is appendable $(e)$-positive as a labelled graph.
\end{remark}

\begin{corollary}
A tree $T$ is $(e)$-positive at vertex $v$ if and only if $T$ is a path with $v$ as one of its endpoints.
\end{corollary}
\begin{proof}
For the forward direction, suppose the contrary. Then let $T$ be {a minimal tree that is} $(e)$-positive at a vertex $v$ such that $T$ is not a path with $v$ as one of its endpoints. By Proposition~\ref{prop:cut}, $v$ cannot be a cut vertex of the tree $T$, and so must be a leaf. Let $w$ denote its unique neighbour. 

Note $T$ is obtained from $T-v$ by concatenating a copy of $K_2$ at $w$ so that the other endpoint of the $K_2$ becomes the vertex $v$ in $T$. By \cite[Lemma 7.5]{GebSag}, if 
$$\y_{T-v\at w}=\sum_{|\lambda|+b=|T|-1}c_{(\lambda,b)}e_{(\lambda,b)}$$ then
$$\y_{T\at v}=\sum_{|\lambda|+b=|T|-1}c_{(\lambda,b)}\left(\frac{b-1}{b}e_{({\lambda\pcup (b)},1)}+\frac{1}{b}e_{(\lambda,b+1)}\right).$$ 

Note for $(\lambda,b)$ satisfying $|\lambda|+b=|T|-1$ {in the above equations} that the coefficient of $e_{(\lambda,b)}$ in $\y_{T-v\at w}$ is $b$ times the coefficient of $e_{(\lambda,b+1)}$ in $\y_{T\at v}$. Since $T$ is $(e)$-positive at $v$, it follows then that $T-v$ is $(e)$-positive at $w$. But then $T-v$ is a smaller tree {that is} $(e)$-positive at a vertex $w$ such that either $T-v$ is not a path or $w$ is not one of its endpoints, contradicting the minimality of $T$.

The reverse direction is given by \cite[Proposition 6.4]{GebSag}, where Gebhard and Sagan show that paths are $(e)$-positive at their endpoints.
\end{proof}

\section{$\UBCQSym$ and acyclic orientations}
\label{sec:UBCQSym}
In {this section} we will define the quotient algebra $\UBCQSym$ of $\NCQSym$ by giving a construction analogous to that of the quotient algebra $\UBCSym$ of $\NCSym$. Working in $\UBCQSym$ will aid us in proving theorems relating the coefficients of $\y_{G\at v}$ in the $e$-basis and acyclic orientations.

A \textit{marked composition} $\hat\alpha$ of $d\in \mathbb Z_{>0}$ is a composition of $d$ with a distinguished part, which we identify by writing a caret above {it}. For example, $(2,\hat2,3)$ is a marked composition of $7$ with underlying composition $(2,2,3)$ and distinguished part chosen to be the second $2$. Given a set composition $\Phi\vDash[d]$ with $d\in \mathbb Z_{>0}$, define $\type(\Phi)$ to be the marked composition with underlying composition $\alpha(\Phi)$ and distinguished part corresponding to the part of $\Phi$ containing $d$, e.g. $\type(13\sepsc 45\sepsc 2)=(2,\hat 2,1)$. When $\Phi$ is the empty set composition, write $\type(\Phi)=(\emptyset)$.

Recall that one basis of $\NCQSym$ consists of the $M_\Phi$ over all set compositions $\Phi$. We can define $\UBCQSym$ first as the free vector space spanned by elements $M_{\type(\Phi)}$ over all set compositions $\Phi$. It is naturally a quotient vector space of $\NCQSym$ via the linear projection map $\projUBC:\NCQSym\to\UBCQSym$ sending each $M_\Phi\mapsto M_{\type(\Phi)}$.

{We define an action of $\Sg_d$ on $\NCQSym^d$ by permuting the positions of the variables, analogous to the action of $\Sg_d$ on $\NCSym^d$. It can be verified from the definitions of the monomial basis of $\NCQSym$ and the action of $\delta\in\Sg_d$ on $\NCQSym^d$ that
$$\delta\circ M_{\Phi}=M_{\delta(\Phi)},$$
where the action of $\delta$ on set compositions of $[d]$ is by permuting the elements of the blocks.}

Note for $\Phi\vDash [d]$ and $\delta\in \Sg_d$ fixing $d$ that $\projUBC(\delta\circ M_\Phi)=\projUBC(M_{\delta(\Phi)})=\projUBC(M_\Phi)$. Extending linearly, for any $f\in\NCQSym^d$ and $\delta\in \Sg_d$ fixing $d$, we also have $\projUBC(\delta\circ f)=\projUBC(f)$. The kernel of $\projUBC$ is given by
$$\ker\projUBC=\spam\{\delta\circ f-f\mid f\in\NCQSym^d,\, \delta\in\Sg_d,\,\delta(d)=d\in\mathbb Z_{>0}\}.$$
For homogeneous $f\in\NCQSym^d$ and $g\in\NCQSym^{d'}$ with $\delta\in \Sg_d$ fixing $d\in\mathbb Z_{>0}$, we have
$$(\delta\circ f -f)g = \delta\circ fg - fg,$$
where in the right-hand side we extend $\delta\in\Sg_d$ to a permutation in $\Sg_{d+d'}$ fixing the last $1+d'$ positions, and
$$g(\delta\circ f - f)=\delta'\circ gf-gf,$$
where $\delta'\in\Sg_{d+d'}$ fixes the first $d'$ positions and permutes positions $1+d',\dots,d+d'$ according to $\delta\in \Sg_d$ (and, in particular, fixes $d+d'$). Extending bilinearly, we see that $\ker\projUBC$ is a two-sided graded ideal of $\NCQSym$, so $\UBCQSym$ is a graded quotient algebra of $\NCQSym$ with
$$\UBCQSym=\NCQSym/\spam\{\delta\circ f-f\mid f\in\NCQSym^d,\, \delta\in\Sg_d,\,\delta(d)=d\in\mathbb Z_{>0}\}.$$
We will write $\UBCQSym^d=\projUBC(\NCQSym^d)$ to denote the homogeneous part of degree $d$ in $\UBCQSym$.

Note that the $\projUBC(m_\pi)$ over set partitions $\pi$ of distinct types are linearly independent in $\UBCQSym$, because $\type(\pi)=(\lambda,b)$ if and only if the coefficient of $M_{(\lambda_1,\dots,\lambda_{\ell(\lambda)},\hat b)}$ is nonzero in the expansion of $\projUBC(m_\pi)$ in the $M$-basis. Therefore $\UBCSym$ is a subalgebra of $\UBCQSym$.

A \textit{labelled composition} of $d\in\mathbb Z_{>0}$ is a pair $(\delta,\alpha)$ consisting of a permutation in $\Sg_d$ and a composition of $d$. We may also write a labelled composition more compactly by writing the permutation in one-line notation and including commas to indicate the parts of the composition. For example, we write $(14,52,3)$ to mean the labelled composition with underlying permutation $14523$ and underlying composition $(2,2,1)$.

Given a labelled composition $(\delta,\alpha)$ of $d$, define the quasisymmetric function in noncommuting variables
$$F_{(\delta,\alpha)} = \sum_{\substack{i_{\delta(1)}\le\dots \le i_{\delta(d)} \\ i_{\delta(j)}<i_{\delta(j+1)}\text{ if }j\in\set(\alpha)}}  x_{i_1}\cdots x_{i_d}.$$
For example,
$$F_{(14,52,3)}=\sum_{i_1\le i_4 < i_5 \le i_2 <i_3} x_{i_1}\cdots x_{i_5}.$$
Note for $\delta,\delta'\in\Sg_d$ that
$$F_{(\delta\delta',\alpha)}=\delta\circ F_{(\delta',\alpha)}.$$

For a set composition $\Phi=\Phi_1\sepsc \cdots\sepsc \Phi_{\ell(\Phi)}$, define
$$\Q_{\Phi} = F_{(\Phi_{1}^r,\dots, \Phi_{\ell(\Phi)}^r)},$$
{where} $\Phi_j^r$ denotes the elements in $j$th part of $\Phi$ written in descending order. As an example, $\Q_{13\sepsc 45\sepsc 2} = F_{(31,54,2)}.$ The $\Q_\Phi$ for $\Phi\vDash [d]$ form a basis for $\NCQSym^d$ by an upper-triangularity argument against the $M_\Phi$, after ordering set compositions by refinement.

Note if set compositions $\Phi,\Psi\vDash[d]$ are of the same type $\hat\alpha$ with underlying composition $\alpha$ and $k$th part marked, then $\projUBC(\Q_\Phi)=\projUBC(\Q_\Psi)$. This follows because $(\Phi_1^r,\dots,\Phi_{\ell(\Phi)}^r)$ and $(\Psi_1^r,\dots,\Psi_{\ell(\Psi)}^r)$ both have underlying composition $\alpha$, and their underlying permutations $\delta,\delta'$ satisfy $\delta(\sum_{j=1}^{k-1} \alpha_j+1)=\delta'(\sum_{j=1}^{k-1} \alpha_j+1)=d$, and so $\delta\delta'^{-1}\in \Sg_d$ fixes $d$, and therefore
$$\projUBC(\Q_\Phi)=\projUBC(F_{(\delta,\alpha)})=\projUBC(\delta\delta'^{-1}\circ F_{(\delta',\alpha)})=\projUBC(F_{(\delta',\alpha)})=\projUBC(\Q_{\Psi}).$$

Define then $\Q_{\type(\Phi)} = \projUBC(\Q_\Phi)$. Because the $\Q_\Phi$ over all set compositions of $[d]$ span $\NCQSym^d$, it follows that the $\Q_{\type(\Phi)}$ over all set compositions of $[d]$ span $\UBCQSym^d$. Moreover, since the dimension of $\UBCQSym^d$ is exactly the number of marked compositions of $d$, it follows that the $\Q_{\type(\Phi)}$ over all set compositions of $d$ form a basis for $\UBCQSym^d$.

A \textit{labelled poset} on $d\in\mathbb Z_{>0}$ elements is a poset given by the set $[d]$ with a partial ordering $<_P$. For a labelled poset $P$ on $d$ elements, define the quasisymmetric function in noncommuting variables
$$Y_P =\sum_{\kappa}x_{\kappa(1)}\dots x_{\kappa(d)},$$ where the sum is over all maps $\kappa:[d]\to \mathbb Z_{>0}$ satisfying $\kappa(i)<\kappa(j)$ whenever $i<_Pj$. Also let $\y_P$ denote $\projUBC(Y_P)$. We can define the action of $\delta\in\Sg_d$ on labelled posets on $d$ elements by permuting labels, e.g. the labelled poset $\delta(P)$ is a relabelling of $P$. It follows by the definition of $Y_P$ and the action of $\delta$ on $\NCQSym^d$ and on labelled posets on $d$ elements that $Y_{\delta(P)}=\delta\circ Y_P$.

We will require the following lemma, which follows by \cite[Lemma 4.5.3(b)]{ECI} from the theory of $P$-partitions.
\begin{lemma}\label{lem:P-part}
Let $P$ be a labelled poset on $d$ elements, and let $s$ be a fixed linear extension of $P$. If $w$ is another linear extension of $P$, given by $i_1<_w\cdots<_w i_d$, then define $\delta_w\in \Sg_d$ to be the permutation $i_1\cdots i_d$, and $\alpha_{w}^s$ to be the composition of $d$ satisfying $\set(\alpha_w^s)=\{j \mid i_j <_s i_{j+1}\}$.

Then
$$Y_P=\sum_w F_{(\delta_w, \alpha_{w}^s)},$$
where the sum is over all linear extensions $w$ of $P$.
\end{lemma}

\begin{corollary}\label{cor:qpos}
Let $P$ be a labelled poset on $d$ elements. Let $s$ be a linear extension of $P$ satisfying $i>_s d$ if and only if $i>_P d$, and let $\eps\in\Sg_d$ be the permutation satisfying
$$\eps\delta_s(i)=\begin{cases}
i&\text{if $i<\delta_s^{-1}(d)$},\\
d &\text{if $i=\delta_s^{-1}(d)$},\\
i-1 &\text{if $i>\delta_s^{-1}(d)$}.\\
\end{cases}$$

Then the coefficient of $\Q_{\hat\alpha}$, where $\hat\alpha$ has underlying composition $\alpha$ and $k$th part marked, in the $\Q$-expansion of $\y_P$ counts the number of linear extensions $w$ of $\eps(P)$ satisfying $\alpha_w^{\eps(s)}=\alpha$ and $\delta_w(\sum_{j=1}^{k-1}\alpha_j+1)=d$. In particular, $\y_P$ is $\Q$-positive.
\end{corollary}

\begin{proof}
Observe that $\eps$ fixes $d$, since $\eps(d)=\eps\delta_s\delta_s^{-1}(d)=d$. Then
$$\y_P=\projUBC(Y_P)=\projUBC(\eps\circ Y_P)=\projUBC(Y_{\eps(P)})=\y_{\eps(P)}.$$
Next consider the linear extension $\eps(s)$ of $\eps(P)$, given by $$1<_{\eps(s)}\cdots <_{\eps(s)} \delta_s^{-1}(d)-1 <_{\eps(s)} d <_{\eps(s)} \delta_s^{-1}(d) <_{\eps(s)}\cdots <_{\eps(s)} d-1.$$ Note that we have $i>_{\eps(s)} d$ if and only if $i>_{\eps(P)}d$.

For any linear extension $w$ of $\eps(P)$, we note $i\not\in \set(\alpha_w^{\eps(s)})$ if and only if $\delta_w(i)>_{\eps(s)} \delta_w(i+1)$, which occurs only if $\delta_w(i)>\delta_w(i+1)$ in $\mathbb Z_{>0}$ or {if} $\delta_w(i+1)=d$. The latter case cannot occur, because $\delta_w(i)>_{\eps(s)}d$ if and only if $\delta_w(i)>_{\eps(P)} d$, and so $\delta_w(i)>_w d$ since $w$ is a linear extension of $\eps(P)$. So $i\not\in\set(\alpha_w^{\eps(s)})$ only if $\delta_w(i)>\delta_w(i+1)$ in $\mathbb Z_{>0}$, implying that the marked composition ${(\delta_w, \alpha_w^{\eps(s)})}$ is of the form $(\Phi_{1}^r,\dots, \Phi_{\ell(\Phi)}^r)$ for some set composition $\Phi\vDash[d]$.

The result then follows by applying Lemma~\ref{lem:P-part} to $\eps(P)$ with the linear extension $\eps(s)$ and studying the projection in $\NCQSym^d$, since then each $F_{(\delta_w,\alpha_w^{\eps(s)})}$ is equal to some $\Q_\Phi$.
\end{proof}

We can now \svw{state} a noncommutative analogue of Stanley's \cite[Theorem 3.3]{Stan95}. \svw{It is proved by adapting Stanley's original proof using $P$-partitions and using the linear map $\varphi:\UBCQSym^{|G|}\to \mathbb Q[t]$ satisfying
$$\varphi(\Q_{\hat\alpha})=\begin{cases}
t(t-1)^{k-1}&\text{if }\hat\alpha=(1^{|G|-k},\hat k),\\
0&\text{otherwise.}
\end{cases}$$ }

\begin{theorem} \label{the:vsink}
Suppose $$\y_{G\at v} = \sum_{|\lambda|+b=|G|}c_{(\lambda,b)}e_{(\lambda,b)}.$$
Let $\sink_v(G,j)$ count the number of acyclic orientations of $G$ with $j$ sinks, including a sink at $v$. Then
$$\sink_{v}(G,j)=\sum_{\substack{|\lambda|+b=|G|\\\ell(\lambda)+1=j}}c_{(\lambda,b)}\lambda!(b-1)!.$$
\end{theorem}

\begin{example}
We will apply Theorem~\ref{the:vsink} to the path $P_4$ and its last vertex. We can compute
$$\y_{P_4}=\frac{1}{3}e_{((3),1)}+\frac{1}{2}e_{((2),2)}+\frac{1}{6}e_{(\emptyset, 4)}.$$ By the theorem, there are exactly $\frac{1}{3}3!0! + \frac{1}{2}2!1! = 3$ acyclic orientations of $P_4$ with two sinks, one of which is at the last vertex, and $\frac{1}{6}3!=1$ acyclic orientation of $P_4$ with a unique sink at the last vertex. These are shown in Figure~\ref{fig:sinks} below.

\begin{figure}[H]
\caption{} 
\label{fig:sinks}
\begin{tikzpicture}
\coordinate (A) at (0,0);
\coordinate (B) at (1,0);
\coordinate (C) at (2,0);
\coordinate (D) at (3,0);
\draw[-{>[scale=1.5]}] (B) -- (A);
\draw[-{>[scale=1.5]}] (C) -- (B);
\draw[-{>[scale=1.5]}] (C) -- (D);

\filldraw[black] (A) circle [radius=2pt];
\filldraw[black] (B) circle [radius=2pt];
\filldraw[black] (C) circle [radius=2pt];
\filldraw[black] (D) circle [radius=2pt] node[below] {\scriptsize $4$};

\coordinate (A) at (5,0);
\coordinate (B) at (6,0);
\coordinate (C) at (7,0);
\coordinate (D) at (8,0);
\draw[-{>[scale=1.5]}] (B) -- (A);
\draw[-{>[scale=1.5]}] (B) -- (C);
\draw[-{>[scale=1.5]}] (C) -- (D);

\filldraw[black] (A) circle [radius=2pt];
\filldraw[black] (B) circle [radius=2pt];
\filldraw[black] (C) circle [radius=2pt];
\filldraw[black] (D) circle [radius=2pt] node[below] {\scriptsize $4$};

\coordinate (A) at (10,0);
\coordinate (B) at (11,0);
\coordinate (C) at (12,0);
\coordinate (D) at (13,0);
\draw[-{>[scale=1.5]}] (A) -- (B);
\draw[-{>[scale=1.5]}] (C) -- (B);
\draw[-{>[scale=1.5]}] (C) -- (D);

\filldraw[black] (A) circle [radius=2pt];
\filldraw[black] (B) circle [radius=2pt];
\filldraw[black] (C) circle [radius=2pt];
\filldraw[black] (D) circle [radius=2pt] node[below] {\scriptsize $4$};
\end{tikzpicture}
\begin{tikzpicture}
\coordinate (A) at (0,0);
\coordinate (B) at (1,0);
\coordinate (C) at (2,0);
\coordinate (D) at (3,0);
\draw[-{>[scale=1.5]}] (A) -- (B);
\draw[-{>[scale=1.5]}] (B) -- (C);
\draw[-{>[scale=1.5]}] (C) -- (D);

\filldraw[black] (A) circle [radius=2pt];
\filldraw[black] (B) circle [radius=2pt];
\filldraw[black] (C) circle [radius=2pt];
\filldraw[black] (D) circle [radius=2pt] node[below] {\scriptsize $4$};
\end{tikzpicture}
\end{figure}
\end{example}

\begin{remark}
Gebhard and Sagan's \cite[Theorem 4.4]{GebSag} is Theorem~\ref{the:vsink} in the special case of $j=1$.
\end{remark}

From Theorem~\ref{the:vsink} we can recover Stanley's original \cite[Theorem 3.3]{Stan95}.

\begin{corollary}
\cite[Theorem 3.3]{Stan95}\label{cor:sink}
Suppose $$X_G=\sum_{\lambda\vdash |G|}c_\lambda e_\lambda.$$ Let $\sink(G,j)$ count the number of acyclic orientations of $G$ with $j$ sinks. Then $$\sink(G,j)=\sum_{\substack{\lambda\vdash |G|\\\ell(\lambda)=j}}c_\lambda.$$
\end{corollary}

\begin{remark}
In \cite{HJLOY}, the acyclic orientation polynomial of a graph is introduced and defined to be the generating function for the sinks of a graph's acyclic orientations. It is shown in \cite[Theorem 3.2]{HJLOY} that the acyclic orientation polynomial satisfies a subgraph expansion analogous to \cite[Theorem 2.5]{Stan95} for $X_G$, and so the authors of \cite{HJLOY} were able to prove the existence of a linear map sending $X_G\mapsto \sum_{j=1}^{|G|}\sink(G,j)t^j$, giving a new proof in \cite[Theorem 4.5]{HJLOY} of Stanley's sink theorem (Corollary \ref{cor:sink}) without using $P$-partitions.

In fact, there is a linear map sending $Y_G$ to the acyclic orientation polynomial of a labelled graph $G$, e.g. by comparing \cite[Theorem 3.6]{GebSag} with \cite[Theorem 3.2]{HJLOY}. After composing that with the linear map sending the acyclic orientation polynomial of $G$ to $\sum_{j=1}^{|G|}\sink_v(G,j)t^j$, one can construct the linear map induced on the quotient, sending $\y_{G\at v}\mapsto \sum_{j=1}^{|G|}\sink_v(G,j)t^j$, which gives another proof of {Theorem~\ref{the:vsink}.}
\end{remark}

\section{{The pointed chromatic symmetric function and further avenues}}
\label{sec:further}
In \cite{Paw}, Pawlowski defines \textit{the pointed chromatic symmetric function} $X_{G,v}\in\Sym[t]$ for a graph $G$ with distinguished vertex $v$. From \cite[Theorem 3.6]{GebSag} and \cite[Definition 3.1]{Paw}, one can show that the injective linear map
\begin{align*}
{\eta}:\Sym[t] &\rightarrow \UBCSym\\
p_{\lambda}t^j &\mapsto p_{(\lambda,j+1)}
\end{align*}
sends each $X_{G,v}\mapsto \y_{G\at v}$. That is, results for $X_{G,v}$ lift to results for $\y_{G\at v}$, and vice-versa. For example, the proof of Theorem~\ref{the:G-v} says that there is a linear map $\vartheta\eta:\Sym[t]\to\Sym$ sending each $X_{G,v}\mapsto X_{G-v}$ by evaluating at $t=0$.

Pawlowski studied the expansion of $X_{G,v}$ in the basis of $\Sym[t]$ consisting of \textit{pointed Schur functions}, first considered by Strahov in \cite{Str}. By applying $\eta$, one can construct a basis of Schur functions for $\UBCSym$ and define $(s)$-positivity and appendable $(s)$-positivity in a way analogous to our definitions of $(e)$-positivity and appendable $(e)$-positivity. Then \cite[Theorem 3.15]{Paw} equivalently states that the paths $P_n$ for $n\ge 1$ are appendable $(s)$-positive. It is known that the $(s)$-positivity of $X_{G,v}$ implies the Schur-positivity of $X_G$. It can be shown that $X_{G-v}$ is also Schur-positive, by applying the map $\vartheta\eta$ together with \cite[Theorem 6.4.1]{Str}.

Another basis of $\Sym[t]$ considered by Pawlowski is the basis of \textit{pointed elementary symmetric functions}. One can deduce by \cite[Definition 4.1]{Paw} and \cite[Theorem 4.6]{Paw} that the image of a pointed elementary symmetric function under $\eta$ is an element in $\UBCSym$ of the form $\frac{1}{\lambda!(b-1)!}e_{(\lambda,b)}$. (Note that the same scaling coefficients arise in the proof of Theorem~\ref{the:G-v} and in \svw{Theorem~\ref{the:vsink}.}) Therefore, $X_{G,v}$ is pointed $e$-positive if and only if $\y_{G\at v}$ is $(e)$-positive, and Pawlowski's \cite[Conjecture 2]{Paw} is equivalent to Conjecture~\ref{conj:alle}. Additionally, by \cite[Corollary 4.3]{Paw}, any $(e)$-positive function is necessarily $(s)$-positive, so for any conjecture of $(e)$-positivity, one can also study the weaker conjecture of $(s)$-positivity for the same functions. 

From the framework introduced in this paper, there are various further avenues of study, especially as possible approaches to the Stanley-Stembridge conjecture. It would be interesting if there existed a $q$-analogue of Conjecture~\ref{conj:alle}. Naively defining a chromatic quasisymmetric function centred at a vertex by a construction similar to \cite[Definition 1.2]{SW} gives a function in $\UBCQSym[q]$ but not $\UBCSym[q]$ even just for the labelled unit interval graph $K_2$. It may, however, be possible to define such a function implicitly for labelled unit interval graphs at their rightmost vertex by requiring the function satisfy a $q$-analogue of the relations in Proposition~\ref{prop:(e)avg} and taking the function evaluated at $K_{\lambda_1}\mid\dots\mid{K_{\ell(\lambda)}}\mid K_b$ to be $\frac{[\lambda_1]_q!\cdots[\lambda_{\ell(\lambda)}]_q![b]_q!}{\lambda!b!}e_{(\lambda,b)}$. Abreu and Nigro showed something similar holds for the chromatic quasisymmetric functions of labelled unit interval graphs in \cite[Theorem 1.1]{AN}.

In the first part of \cite[Conjecture 6.1]{DSvWposets}, Dahlberg, She and van Willigenburg conjectured that for any connected labelled unit interval graph $G$, if one constructs a second labelled unit interval graph $G'$ by following a certain procedure, then $G\ge_e G'$ in the chromatic $e$-positivity poset. In \cite[Remark 6.3]{DSvWposets}, they note that this conjecture would imply the Stanley-Stembridge conjecture. It may be fruitful to also study either a $q$-analogue of this newer conjecture or a version of it in $\UBCSym$ (or a version combining the two, if an appropriate definition for the chromatic quasisymmetric function centred at a vertex exists for labelled unit interval graphs at their rightmost vertex). The relations in Corollary~\ref{cor:eavg} and Proposition~\ref{prop:(e)avg}, as well as their $q$-analogues, can help reduce the number of cases needed to prove various versions of the conjecture of Dahlberg, She and van Willigenburg.

\section{Acknowledgments}\label{sec:ack} The authors would like to thank Mathieu Guay-Paquet, Bruce Sagan and John Shareshian for helpful conversations, {and in particular, Samantha Dahlberg for this and sharing her code}. \svw{They also thank the referee for their thoughtful comments.}

\bibliographystyle{plain}

\begin{thebibliography}{10}

\bibitem{AN} A.~Abreu and A.~Nigro, Chromatic symmetric functions from the modular law, \emph{J. Combin. Theory Ser. A} {180}, 105407 (2021). 

\bibitem{Jose2} J.~Aliste-Prieto and J.~Zamora, Proper caterpillars are distinguished by their chromatic symmetric function, \emph{Discrete Math.} 315, 158--164 (2014).

\bibitem{Birk} G.~Birkhoff, A determinant formula for the number of ways of coloring a map, \emph{Ann. of Math.} 14, 42--46
(1912).

\bibitem{BC} P.~Brosnan and T.~Chow, Unit interval orders and the dot action on the cohomology of regular semisimple Hessenberg varieties, \emph{{Adv. Math.}} 329, 955--1001 (2018).

\bibitem{Dladders}
{S.~Dahlberg, Triangular ladders $P_{d,2}$ are $e$-positive, {\tt arXiv:1811.04885v2}}

\bibitem{DSvWcut}
S.~Dahlberg, A.~She and S.~van Willigenburg, Schur and $e$-positivity of trees and cut vertices, \emph{Electron. J. Combin.} 27, P1.2 22pp (2020).

\bibitem{DSvWposets}
S.~Dahlberg, A.~She and S.~van Willigenburg, Chromatic posets, \emph{J. Combin. Theory Ser. A} 184, 105496 (2021).


\bibitem{DvWpop}
S.~Dahlberg and S.~van Willigenburg, Lollipop and lariat symmetric functions, \emph{SIAM J. Discrete Math.} 32, 1029--1039 (2018).

\bibitem{DvWNC}
S.~Dahlberg and S.~van Willigenburg, Chromatic symmetric functions in noncommuting variables revisited, \emph{Adv. in Appl. Math.} 112{,} 25pp (2020).

\bibitem{GebSag} D.~Gebhard and B.~Sagan, A chromatic symmetric function in noncommuting variables, \emph{J. Algebraic Combin.} 13, 227--255 (2001).

\bibitem{MGP} {M.~Guay-Paquet, A modular relation for the chromatic symmetric functions of ${(3+1)}$-free posets, {\tt arXiv:1306.2400v1}}

\bibitem{MGP2} {M.~Guay-Paquet, A second proof of the Shareshian-Wachs conjecture, by way of a new Hopf algebra, {\tt arXiv:1601.05498v1}}

\bibitem{MM} M.~Harada and M.~Precup, The cohomology of abelian {H}essenberg
varieties and the Stanley-Stembridge conjecture, \emph{Algebraic Combinatorics} 2, 1059--1108 (2019).

\bibitem{HeilJi}
S.~Heil and C.~Ji, On an algorithm for comparing the chromatic symmetric functions of trees, \emph{Australas. J. Combin.} 75, 210--222 (2019).

\bibitem{HuhNamYoo}
J.~Huh, S.-Y.~Nam and M.~Yoo, Melting lollipop chromatic quasisymmetric functions and {S}chur expansion of unicellular {LLT} polynomials, \emph{Discrete Math.} 343, 111728 (2020).

\bibitem{HJLOY}
B.-H.~Hwang, W.-S.~Jung, K.-J.~Lee, J.~Oh and S.-H.~Yu, Acyclic orientation polynomials and the sink theorem for chromatic symmetric functions, \emph{J. Combin. Theory Ser. B} 149, 52--75 (2021).

\bibitem{LS}
M.~Loebl {and} J.-S. Sereni, Isomorphism of weighted trees and {S}tanley's isomorphism conjecture for caterpillars, \emph{{Ann. Inst. Henri Poincar\'{e}} D} 6, 357--384 (2019).

\bibitem{Mac} I.~Macdonald, Symmetric Functions and Hall Polynomials, \emph{Oxford University Press}, edition 2, (1995).

\bibitem{MMW} J.~Martin, M.~Morin and J.~Wagner, On distinguishing trees by their chromatic symmetric
functions, \emph{J. Combin. Theory Ser. A} 115 {(2)}, 237--253 (2008).

\bibitem{OS} R.~Orellana and G.~Scott, Graphs with equal chromatic symmetric function, \emph{Discrete Math.} 320, 1--14 (2014).


\bibitem{Paw} B. Pawlowski, Chromatic symmetric functions via the group algebra of $S_n$, \svw{\emph{Algebr. Comb.}  5, 1--20 (2022).}

\bibitem{RS} M.~Rosas and B.~Sagan, Symmetric functions in noncommuting variables, \emph{Trans. Amer. Math. Soc.} 358, 215--232 (2006).

\bibitem{SW} J.~Shareshian and M.~Wachs, Chromatic quasisymmetric functions, \emph{Adv. Math.} 295, 497--551 (2016).

\bibitem{ECI} R.~Stanley, Enumerative Combinatorics Vol. 1, \emph{Cambridge University Press}, edition {1, (1986).}

\bibitem{Stan95} R.~Stanley, A symmetric function generalization of the chromatic polynomial of a graph, \emph{Adv. Math.} 111, 166--194 (1995).

\bibitem{StanStem} R.~Stanley and J.~Stembridge, On immanants of {J}acobi-{T}rudi matrices and permutations with restricted position, \emph{J. Combin. Theory Ser. A} 62, 261--279 (1993).


\bibitem{Str} E.~Strahov, Generalized characters of the symmetric group, \emph{Adv. Math.} 212, 109--142 (2007).


\end{thebibliography}

\end{document}